\documentclass[a4paper,10pt,reqno]{amsart}

\usepackage[utf8]{inputenc}
\usepackage[T1]{fontenc}
\usepackage{amsthm}
\usepackage{amsmath}
\usepackage{amssymb}
\usepackage[inline]{enumitem}
\usepackage[dvips]{graphicx}
\usepackage{color}
\usepackage{comment}
\usepackage{hyperref}
\usepackage{url}
\usepackage{scalerel}
\usepackage{fancyhdr}
\usepackage{mathrsfs}
\usepackage{stmaryrd}
\usepackage[normalem]{ulem}

\newtheorem{theorem}{Theorem}[section]
\newtheorem{lemma}[theorem]{Lemma}
\newtheorem{proposition}[theorem]{Proposition}
\newtheorem{corollary}[theorem]{Corollary}
\newtheorem{claim}{\textsc{Claim}}

\theoremstyle{definition}
\newtheorem{definition}[theorem]{Definition}
\newtheorem{example}[theorem]{Example}
\newtheorem{remark}[theorem]{Remark}

\definecolor{blue-url}{RGB}{0,0,100}
\definecolor{red-url}{RGB}{100,0,0}
\definecolor{green-url}{RGB}{0,100,0}

\hypersetup{
    pdftitle={Factorization in monoids and rings},
    pdfauthor={Salvatore Tringali},
    pdfmenubar=false,
    pdffitwindow=true,
    pdfstartview=FitH,
    colorlinks=true,
    linkcolor=blue-url,
    citecolor=green-url,
    urlcolor=red-url
}

\renewcommand{\emptyset}{\varnothing}
\renewcommand{\setminus}{\smallsetminus}
\renewcommand{\,}{\kern 0.1em}

\providecommand{\RR}{\mathbin{R}}
\providecommand\llb{\llbracket}
\providecommand\rrb{\rrbracket}

\newcommand{\Conv}{\mathop{\scalebox{2.5}{\raisebox{-0.2ex}{$\ast$}}}}%

\makeatletter
\DeclareFontFamily{OMX}{MnSymbolE}{}
\DeclareSymbolFont{MnLargeSymbols}{OMX}{MnSymbolE}{m}{n}
\SetSymbolFont{MnLargeSymbols}{bold}{OMX}{MnSymbolE}{b}{n}
\DeclareFontShape{OMX}{MnSymbolE}{m}{n}{
	<-6>  MnSymbolE5
	<6-7>  MnSymbolE6
	<7-8>  MnSymbolE7
	<8-9>  MnSymbolE8
	<9-10> MnSymbolE9
	<10-12> MnSymbolE10
	<12->   MnSymbolE12
}{}
\DeclareFontShape{OMX}{MnSymbolE}{b}{n}{
	<-6>  MnSymbolE-Bold5
	<6-7>  MnSymbolE-Bold6
	<7-8>  MnSymbolE-Bold7
	<8-9>  MnSymbolE-Bold8
	<9-10> MnSymbolE-Bold9
	<10-12> MnSymbolE-Bold10
	<12->   MnSymbolE-Bold12
}{}

\let\llangle\@undefined
\let\rrangle\@undefined
\DeclareMathDelimiter{\llangle}{\mathopen}%
{MnLargeSymbols}{'164}{MnLargeSymbols}{'164}
\DeclareMathDelimiter{\rrangle}{\mathclose}%
{MnLargeSymbols}{'171}{MnLargeSymbols}{'171}
\makeatother

\begin{document}

\title{Factorization in Monoids and Rings}
\author{Salvatore Tringali}
\address{School of Mathematical Sciences,
Hebei Normal University | Shijiazhuang, Hebei province, 050024 China}
\email{salvo.tringali@gmail.com}
\urladdr{http://imsc.uni-graz.at/tringali}
%
%
%
%

%
%
\begin{abstract}
\noindent{}
Let $H^\times$ be the group of units of a multiplicatively written monoid $H$. We say $H$ is acyclic if $xyz \ne y$ for all $x, y, z \in H$ with $x \notin H^\times$ or $z \notin H^\times$; unit-cancellative if $yx \ne x \ne xy$ for all $x, y \in H$ with $y \notin H^\times$; f.g.u. if there is a finite set $A \subseteq H$ such that every non-unit of $H$ is a finite product of elements of the form $uav$ with $u, v \in H^\times$ and $a \in A$; l.f.g.u. if, for each $x \in H$, the smallest divisor-closed submonoid of $H$ containing $x$ is f.g.u; and atomic if every non-unit can be written as a finite product of atoms, where an atom is a non-unit that does not factor into a product of two non-units.

We generalize to l.f.g.u. or acyclic l.f.g.u. monoids a few results so far only known for unit-cancellative l.f.g.u. commutative monoids (cancellative monoids are unit-cancellative, and a commutative monoid is unit-cancellative if and only if it is acyclic). In particular, we prove the following:
\begin{itemize}
\item If $H$ is an atomic l.f.g.u. monoid, then every non-unit has only finitely many factorizations (into atoms) that are ``minimal'' and ``pairwise non-equivalent'' (with respect to some naturally defined rel\-ations on the free monoid over the ``alphabet'' of atoms).
\item If $H$ is an acyclic l.f.g.u. monoid, then it is atomic; and moreover, each element has only finitely many ``pairwise non-equivalent'' factorizations if we additionally assume $H$ to be commutative.
\end{itemize}
\end{abstract}
\maketitle
\thispagestyle{empty}

\section{Introduction}
\label{sec:intro}
By the fundamental theorem of arithmetic, every positive integer other than $1$ can be expressed as a non-empty product of primes in an essentially unique way. Factorization theory is, on the whole, the study of various phenomena related to the possibility or impossibility of extending such a decomposition to arbitrary rings and monoids (see \S{ }\ref{sec:preliminaries} for basic terminology).

Over the years, the field has branched out into many subfields, ranging from
commutative and non-commutative algebra to semigroup theory, from additive combinatorics to the abstract theory of zeta functions. Detailed information
about problems, methods, results, and trends can be found in the conference proceedings \cite{And97,ChGl00,Ch05,ChFoGeOb16}, in the surveys \cite{BaCh11,BaWi13,Ge16c,GeZh19}, and in the mon\-o\-graphs \cite{Nark74,GeHK06,FoHoLu13}.

So far, research in the area has been centered on rings
and monoids --- e.g., monoids of modules, Mori domains, Krull monoids, rings of integer-valued polynomials, monoids of ideals, orders in central simple algebras, monoids of matrices, and numerical monoids ---,
where the structures in play are commutative or cancellative. It is only recently \cite{Fa-Tr18, An-Tr18} that some key aspects of factorization theory have been systematically extended to possibly non-commutative and non-cancellative monoids (and thence to rings that need not be domains), so as
to widen the spectrum of potential applications and foster interaction with other fields.
In the present paper, we further contribute to this line of research.

In particular, we say that a monoid $H$ is \emph{atomic} if every non-unit element of $H$ is a product of atoms (Definition \ref{def:2.2}); \emph{l.f.g.u.} if, for every $x \in H$, the smallest divisor-closed submonoid of $H$ containing $x$ is, up to units, finitely generated (Definition \ref{def:fg-monoids-and-the-like}); and \emph{acyclic} if certain ``cyclic relations'' are forbidden in $H$ (Definition \ref{def:acyclic}). Among others, we will prove the following results, so far only known for commutative and ``nearly canc\-ellative'' monoids (and rings):
\begin{itemize}
\item An atomic l.f.g.u. monoid is FmF (Corollary \ref{cor:lfgu-is-FmF}), that is, every non-unit has only finitely many factorizations into atoms that are ``minimal'' and ``pairwise non-equivalent'' with respect to some naturally defined rel\-ations on the free monoid over the ``alphabet'' of atoms.
\item An acyclic l.f.g.u. monoid is atomic (Corollary \ref{cor:lfgu-acyclic-mons-are-atomic}), and is actually FF (meaning that each element has only finitely many ``pairwise non-equivalent'' factorizations into atoms) if we additionally assume that the monoid is commutative (Corollary \ref{cor:comm-unit-canc-lfgu-mons-are-FF}).
\end{itemize}
In addition, we will establish a characterization of ``unique factorization monoids'' (Theorem \ref{th:factoriality}) in terms of a special type of atoms we call \emph{powerful} (Definition \ref{def:powerful-atoms}), whose relation with primes (Definition \ref{def:prime}) and the fundamental theorem of arithmetic is clarified by Propositions \ref{prop:primes} and \ref{prop:5.8} and Example \ref{exa:5.10}.

The paper closes with some ideas for further research (\S{ }\ref{sec:6}) and includes many stubs that shall help, we hope, to shed light on some delicate points (e.g., Examples \ref{exa:non-atomic-2-generator-1-relator-canc-mon}, \ref{exa:atomic-2-gen-1-rel-canc-monoid}, \ref{exa:5.4}, and \ref{exa:5.7}).

\section{Preliminaries.}
\label{sec:preliminaries}
In this section, we establish some notations and terminology used all through the paper and prepare the ground for the study of l.f.g.u. and acyclic monoids in \S\S{ }\ref{sec:lfgu-monoids} and \ref{sec:acyclic-monoids}. Further terminology and notations, if not explained when first introduced, are standard or should be clear from context.

\subsection{Generalities}\label{sec:generalities}
We use $\mathbf N$ for the non-negative integers, $\bf Z$ for the integers, $\mathbf Q$ for the rationals, and $\mathbf R$ for the reals. For all $a, b \in \mathbf R$, we let $\llb a, b \rrb := \{x \in \mathbf Z: a \le x \le b\}$ be the \emph{discrete interval} between $a$ and $b$. Unless noted otherwise, we reserve the letters $m$ and $n$ (with or without subscripts) for positive integers, and the letters $i$, $j$, $k$, and $\ell$ for non-negative integers.

Given a set $X$ and an integer $k \ge 0$, we write $X^{\times k}$ for the Cartesian product of $k$ copies of $X$; and if $R$ is a binary relation on $X$ and $x$ is an element of $X$, we let
$\llb x \rrb_R := \{y \in X: x \RR y\} \subseteq X$, where $x \RR y$ is a shorthand notation for $(x,y) \in R$. A (\emph{partial}) \emph{preorder} on $X$ is a binary relation $R$ on $X$ such that $x \RR x$ for all $x \in X$ (i.e., $R$ is \emph{reflexive}), and $x \RR z$ whenever $x \RR y$ and $y \RR z$ (i.e., $R$ is \emph{transitive}).
An \emph{order} on $X$ is a preorder $R$ on $X$ such that, if $x \RR y$ and $y \RR x$, then $x = y$; and an \emph{equivalence} (relation) on $X$ is a preorder $R$ on $X$ such that $x \RR y$ if and only if $y \RR x$.
If $R$ is a preorder (resp., an order) on $X$, then we call the pair $(X, R)$ a \emph{preset} (resp., a \emph{poset}).

We will commonly denote a preorder by the symbol $\preceq$ (with or without subscripts or superscripts), and reserve the symbol $\leq$ (and its ``dual'' $\geq$) for the usual order on the reals and its subsets. Accordingly, we shall write $x \prec y$ if $x \preceq y$ and $y \not\preceq x$; notice that ``$x \prec y$'' is a stronger condition than ``$x \preceq y$ and $x \ne y$'', and the two conditions are equivalent when $\preceq$ is an order.

If $(X, \preceq)$ is a preset and $Y$ is a subset  of $X$, we let a \emph{$\preceq$-minimal element} of $Y$ be an element $\bar{y} \in Y$ with the property that there exists no element $y \in Y$ such that $y \prec \bar{y}$.

\subsection{Monoids.}
\label{subsec:fund-defs}
Throughout, monoids will be usually written multiplicatively and, unless a statement to the contrary is made, need not have any special property (e.g., commutativity). 

Let $H$ be a monoid with identity $1_H$. We denote by $H^\times$ and $\mathscr A(H)$, resp., the set of \emph{units} and the set of \emph{atoms} of $H$, where $u \in H$ is a unit if $uv = vu = 1_H$ for some provably unique $v \in H$, called the \emph{inverse} of $u$ (in $H$) and denoted by $u^{-1}$; and $a \in H$ is an atom if $a \notin H^\times$ and $a \ne xy$ for all $x, y \in H \setminus H^\times$. It is easily seen that $H^\times$ is a subgroup of $H$, hence referred to as the \emph{group of units} of $H$. 

A (monoid) \emph{congruence} on $H$ is an equivalence relation $R$ on $H$ such that, if $x \RR u$ and $y \RR v$, then $xy \RR uv$. Note that, if $R$ is a congruence on $H$, we will usually write $x \equiv y \bmod R$ in place of $x \RR y$.

We say $H$ is \emph{reduced} if $H^\times = \{1_H\}$; \emph{cancellative} if $xz = yz$ or $zx = zy$, for some $x, y, z \in H$, implies $x = y$; and \emph{unit-cancellative} if $xy \ne x \ne yx$ for all $x, y \in H$ with $y \notin H^\times$.

Unit-cancellative monoids have been the subject of several recent papers in fac\-tor\-i\-za\-tion theory, both in the commutative \cite{FGKT,Ge-Re19} and in the non-commutative setting \cite{Ge-Schw18,Fa-Tr18,An-Tr18}; and among others, it is obvious that a cancellative monoid is unit-cancellative (see also Remark \ref{rem:history}).

Given $x \in H$, we call
$\text{ord}_H(x) := |\{x^n: n \in \mathbf N^+\}| \in \mathbf N \cup \{\infty\}$
the \emph{order} of $x$ (relative to $H$); and we say that $x$ is an element of finite order (in $H$) if $\text{ord}_H(x) < \infty$, an \emph{idempotent} if $x^2 = x$, and a \emph{non-trivial idempotent} if $x^2 = x \ne 1_H$. For all $X_1, \ldots, X_n \subseteq H$, we denote by
\begin{equation*}\label{equ:def-of-setwise-multiplication}
X_1 \cdots X_n := \{x_1 \cdots x_n: x_1 \in X_1, \ldots, x_n \in X_n\} \subseteq H
\end{equation*}
the \emph{setwise product} of $X_1, \ldots, X_n$ (relative to $H$); and by abuse of notation we replace $X_i$ with $x_i$ on the left of the symbol ``:='' 
in the above definition
if $X_i = \{x_i\}$ for some $i \in \llb 1, n \rrb$ and there is no risk of confusion. In particular, for every $X \subseteq H$ we let
$$
\mathrm{Sgp}\langle X \rangle_H := \bigcup_{n \ge 1} X^n = \{x_1 \cdots x_n: x_1, \ldots, x_n \in X \} \subseteq H
$$
be the \emph{subsemigroup} of $H$ \emph{generated by $X$}, where $X^n$ is the setwise product of $n$ copies of $X$ (note that $1_H$ need not be in $ \mathrm{Sgp}\langle X \rangle_H$, and hence $\mathrm{Sgp}\langle X \rangle_H$ need not be a submonoid of $H$). We shall commonly write 
$
\mathrm{Sgp}\langle x_1, \ldots, x_n \rangle_H$
instead of $\mathrm{Sgp}\langle X \rangle_H$ when $X$ is a non-empty finite set with elements $x_1, \ldots, x_n$.

We denote by $\mid_H$ and $\simeq_H$, resp., the binary relations on $H$ defined by: $x \mid_H y$ if and only if $x \in HyH$; and $x \simeq_H y$ if and only if $x \in H^\times y H^\times$. 
We will often use without comment that $\mid_H$ is a preorder and $\simeq_H$ is an equivalence relation (on $H$); and use the symbols $\nmid_H$ and $\not\simeq_H$ with their obvious meaning.

With the above in place, we take a \emph{divisor-closed submonoid} of $H$ to be a sub\-monoid $M$ of $H$ such that, if $x \mid_H y$ and $y \in M$, then $x \in M$; and for $X \subseteq H$ we set
$$
\llangle X \rrangle_H := \bigcap \left\{M: X \subseteq M \text{ and }M \text{ is a divisor-closed submonoid of } H\right\} \subseteq H.
$$
Observe that $\llangle X \rrangle_H$ is a divisor-closed submonoid of $H$ containing $X$, whence we call $\llangle X \rrangle_H$ the \emph{divisor-closed submonoid of $H$ generated by $X$}. We shall commonly write
$
\llangle x_1, \ldots, x_n \rrangle_H 
$
instead of $\llangle X \rrangle_H$ when $X$ is a non-empty finite set with elements $x_1, \ldots, x_n$.

\subsection{Free monoids.}
\label{sec:free-monoids}
Let $X$ be a set, in this context often dubbed as an ``alphabet''. We will write $\mathscr F(X)$ for the free monoid over $X$; and refer to an element of $\mathscr F(X)$ as an \emph{$X$-word}, or simply as a \emph{word} if no confusion can arise. We shall use the symbols $\ast_X$ and $\varepsilon_X$, resp., for the operation and the identity of $\mathscr F(X)$; and drop the subscript ``$X$'' from this notation when there is no risk of ambiguity.

We recall that $\mathscr F(X)$ consists, as a set, of all \emph{finite} tuples of elements of $X$; and $\mathfrak u \ast_X \mathfrak v$ is the \emph{con\-cat\-e\-na\-tion} of two such tuples $\mathfrak u$ and $\mathfrak v$. Accordingly, it is found that the identity of $\mathscr F(X)$ is the ``empty tuple'' of elements of $X$, which, in dealing with free monoids, we rather call the \emph{empty $X$-word}. 

For technical reasons, we will assume that, if $Y \subseteq X$, then $\mathscr F(Y) \subseteq \mathscr F(X)$. In particular, this implies that the ``empty tuple'' of elements of $X$ has been implicitly defined in a way that it does not depend on the choice of $X$, as is possible to do, e.g., in Tarski-Grothendieck set theory (by requiring that all sets belong to some fixed universe) and other axiomatic set theories.

We take the \emph{length} of an $X$-word $\mathfrak u$, denoted by $\|\mathfrak u\|_X$, to be the unique non-negative integer $k$ such that $\mathfrak u \in X^{\times k}$; in particular, the empty word is the only $X$-word whose length is zero. Notice that, if $\mathfrak u$ is an $X$-word of positive length $k$, then there are de\-ter\-mined $u_1, \ldots, u_k \in X$ for which $\mathfrak u = u_1 \ast \cdots \ast u_k$. 

Given $\mathfrak u, \mathfrak v \in \mathscr F(X)$, we set $
\mathfrak u^{\ast 0} := \varepsilon$, 
$\mathfrak u^{\ast 1} := \mathfrak u$, and $
\mathfrak u^{\ast n} := \mathfrak u^{\ast(n-1)} \ast \mathfrak u$ for all $n \in \mathbf N^+$; and we say that  $\mathfrak u$ is a \emph{subword} of $\mathfrak v$ if either $\mathfrak u$ is empty, or $\mathfrak v$ is a non-empty word of length $\ell$, say $\mathfrak v = v_1 \ast \cdots \ast v_\ell$, and $\mathfrak u = v_{i_1} \ast \cdots \ast v_{i_k}$ for some $k \in \mathbf N^+$ and $i_1, \ldots, i_k \in \llb 1, \ell \rrb$ such that $i_j < i_{j+1}$ for each $j \in \llb 1, k-1 \rrb$.

Our interest in words and subwords is related, on the one hand, to the notion of ``factorization'' (see \S{ }\ref{subsec:factorizations}) and, on the other, to the following combinatorial result, which is commonly known as Higman's lemma and will be a main ingredient in the proof of Corollary \ref{cor:lfgu-is-FmF}.

\begin{lemma}\label{lem:higman-lemma}
	If $X$ is a finite alphabet and $(\mathfrak u_i)_{i \ge 1}$ is an infinite sequence of $X$-words, then there exist $i, j \in \mathbf N^+$ with $i \ne j$ such that $\mathfrak u_i$ is a subword of $\mathfrak u_j$.
\end{lemma}
This is essentially a special case of \cite[Theorem 4.4]{Hi52} and can be thought of as a ``non-commutative generalization'' of Dickson's lemma, another combinatorial result, usually attributed to L.E. Dickson, which has been crucial to the study of the arithmetic of integral domains and ``nearly cancellative'' commutative monoids (see \cite[Theorem 2.9.13]{GeHK06}, \cite[Proposition 7.3]{Ge16c}, and \cite[Proposition 3.4]{FGKT} for some representative results in this direction).

\subsection{Factorizations.} 
\label{subsec:factorizations}
Most of the contents of this section are borrowed from \cite{Fa-Tr18} and \cite{An-Tr18}, where one can read extensively about differences and similarities with 
alternative approaches to the study of factorization in settings (namely, commutative rings and cancellative or com\-mu\-ta\-tive monoids) that are, however, less general than ours; in particular,
see \cite[Remarks 2.6 and 2.7]{Fa-Tr18} and \cite[\S{ }2.4 and Remarks 4.4 and 4.5]{An-Tr18}.

Let $H$ be a monoid. We write $\pi_H$ for the unique monoid homomorphism from $\mathscr F(H)$ to $H$ such that $\pi_H(a) = a$ for every atom $a \in H$ 
(we call $\pi_H$ the \emph{factorization homomorphism} of $H$), and we denote by $\mathscr{C}_H$ the monoid con\-gru\-ence on $\mathscr F(\mathscr{A}(H))$ defined by: $\mathfrak a \equiv \mathfrak b \bmod \mathscr C_H$ if and only if either $\mathfrak a = \mathfrak b = \varepsilon$; or $\mathfrak a$ and $\mathfrak b$ are non-empty $\mathscr A(H)$-words of length $n$, say $\mathfrak a = a_1 \ast \cdots \ast a_n$ and $\mathfrak b = b_1 \ast \cdots \ast b_n$, such that $\pi_H(\mathfrak a) = \pi_H(\mathfrak a)$ and $a_1 \simeq_H b_{\sigma(1)}$, \ldots, $a_n \simeq_H b_{\sigma(n)}$ for some permutation $\sigma$ of the discrete interval $\llb 1, n \rrb$.

\begin{definition}
\label{def:2.2}
Given $x \in H$, we take the \emph{set of factorizations} of $x$ (into atoms of $H$) to be the set
\[
\mathcal{Z}_H(x) := \pi_H^{-1}(x) \cap \mathscr F(\mathscr{A}(H));
\]
and then we refer to the sets 
\[
{\sf L}_H(x) := \big\{\|\mathfrak a\|_{\mathscr A(H)}: \mathfrak a \in \mathcal{Z}_H(x)\big\} \subseteq \mathbf N
\quad\text{and}\quad
\mathsf{Z}_H(x) := \bigl\{ \llb \mathfrak{a} \rrb_{\mathscr{C}_H} : \mathfrak{a}\in \mathcal{Z}_H(x) \bigr\}\subseteq \mathscr F(\mathscr{A}(H))/\mathscr{C}_H,
\]
resp., 
as the \emph{set of lengths} and the \emph{set of factorization classes} of $x$ (relative to the atoms of $H$).

Accordingly, we say that the monoid $H$ is 
\begin{itemize}
\item \emph{atomic} if $\mathcal{Z}_H(x)$ is non-empty for every $x \in H \setminus H^\times$; 
\item BF (resp., FF) if $H$ is atomic and $\mathsf{L}_H(x)$ (resp., $\mathsf{Z}_H(x)$) is finite for each $x \in H$; 
\item \textup{HF} or \emph{half-factorial} (resp., \emph{factorial}) if $|\mathsf L_H(x)| = 1$ (resp., $|\mathsf{Z}_H(x)| = 1$) for every $x \in H \setminus H^\times$.
\end{itemize}
\end{definition}
The above definitions are modeled after A. Geroldinger and F. Halter-Koch's monograph \cite{GeHK06} on factorization in integral domains and cancellative commutative monoids. However, factorizations and sets of lengths in non-commutative or non-cancellative monoids tend to ``blow up'' in a predictable way (due, e.g., to the presence of non-trivial idempotents), with the result that most of the invariants studied in the ``classical theory'' lose their significance. To counter this phenomenon --- inherent to the structures studied in the present work ---, we adopt the approach set forth in \cite[\S{ }4]{An-Tr18}.

We denote by $\preceq_H$ 
the binary relation on $\mathscr F^\ast(\mathscr A(H))$ defined by: $\mathfrak a \preceq_H \mathfrak b$ if and only if either $\mathfrak a = \mathfrak b = \varepsilon$; or $\mathfrak a$ and $\mathfrak b$ are non-empty $\mathscr A(H)$-words of length $m$ and $n$ resp., say $\mathfrak a = a_1 \ast \cdots \ast a_m$ and $\mathfrak b = b_1 \ast \cdots \ast b_n$, such that $\pi_H(\mathfrak a) = \pi_H(\mathfrak b)$ and
$a_{\sigma(1)} \simeq_H b_1$, \ldots, $a_{\sigma(m)} \simeq_H b_m$ for some injection $\sigma: \llb 1, m \rrb \to \llb 1, n \rrb$. 
It is easily found that $\preceq_H$ is a preorder on $\mathscr F(\mathscr A(H))$, see \cite[Proposition 4.2(i)]{An-Tr18}. This leads to the following:

\begin{definition}
Given $x \in H$, we let a \emph{$\preceq_H$-minimal factorization} of $x$ be a $\preceq_H$-minimal $\mathscr A(H)$-word $\mathfrak a$ such that $x = \pi_H(\mathfrak a)$. Accordingly, we take the \emph{set of $\preceq_H$-minimal factorizations} of $x$ to be the set
\[
\mathcal{Z}_{H}^{\sf m}(x) := \left\{ \mathfrak{a}\in \mathcal Z_H(x): \mathfrak a \textrm{ is $\preceq_H$-minimal} \right\};
\]
and then we refer to the sets
\[
\mathsf{L}_{H}^{\sf m}(x) := \left\{ \|\mathfrak{a}\|_{\mathscr A(H)}: \mathfrak{a} \in \mathcal{Z}_{H}^{\sf m}(x) \right\} \subseteq \mathsf L_H(x) 
\quad\text{and}\quad
\mathsf{Z}_{H}^{\sf m}(x):= \bigl\{ \mathcal Z_H^{\sf m}(x) \cap \llb \mathfrak{a} \rrb_{\mathscr{C}_H}: \mathfrak{a} \in \mathcal{Z}_H(x) \bigr\},
\]
resp., as the \emph{set of $\preceq_H$-minimal factorizations} and the \emph{set of $\preceq_H$-minimal factorization classes} of $x$.

Accordingly, we say that the monoid $H$ is 
\begin{itemize}
\item \emph{\textup{BmF}} (resp., \emph{\textup{FmF}}) if $H$ is atomic and $\mathsf{L}_{H}^{\sf m}(x)$ (resp., $\mathsf Z_H^{\sf m}(x)$) is finite for each $x\in H$;
\item \textup{HmF} (resp., \emph{minimally factorial}) if $|\mathsf L_H^{\sf m}(x)| = 1$ (resp., $|\mathsf{Z}_H^{\sf m}(x)| = 1$) for every $x \in H \setminus H^\times$.
\end{itemize}
\end{definition}
Propositions \ref{prop:atomic-iff-min-atomic} and \ref{prop:divisor-closed-sub} below will help to clarify some of the notions we have introduced so far; in particular, the latter generalizes \cite[Proposition 1.2.11.1]{GeHK06} in showing that the arithmetic of a monoid is controlled, to some extent, by the arithmetic of its one-generated divisor-closed submonoids. We start with a lemma that will often come in handy later on (see \S{ }\ref{subsec:fund-defs} for terminology).

\begin{lemma}
\label{lem:2.4}
Let $H$ be a monoid, and let $a \in \mathscr A(H)$ and $x, y \in H$. The following hold:
\begin{enumerate}[label=\textup{(\roman{*})}]
\item\label{lem:2.4(i)} $uav \in \mathcal A(H)$ for all $u, v \in H^\times$.
\item\label{lem:2.4(ii)} If $\mathscr A(H)$ is non-empty or $H$ is commutative or unit-cancellative, then $xy$ is a unit if and only if $x$ and $y$ are both units.
\end{enumerate} 
\end{lemma}

\begin{proof}
It is obvious that, if $H$ is commutative and $xy$ is a unit for some $x, y \in H$, then $x, y \in H^\times$. For the rest, see \cite[Lemma 2.2, parts (i) and (ii); and Proposition 2.30]{Fa-Tr18}.
\end{proof}

\begin{proposition}\label{prop:atomic-iff-min-atomic}
	Let $H$ be a monoid. The following are equivalent:
	\begin{enumerate}[label=\textup{(\alph{*})}]
	\item \label{it:prop:atomic-iff-min-atomic(a)} $H$ is atomic.
	\item \label{it:prop:atomic-iff-min-atomic(b)}
	$H \setminus H^\times = \mathrm{Sgp} \langle \mathscr A(H) \rangle_H$.
	\item \label{it:prop:atomic-iff-min-atomic(c)}
	Every non-unit of $H$ has at least one $\preceq_H$-minimal factorization.
	\end{enumerate}
\end{proposition}

\begin{proof}
	The implications \ref{it:prop:atomic-iff-min-atomic(b)} $\Rightarrow$ \ref{it:prop:atomic-iff-min-atomic(a)} and \ref{it:prop:atomic-iff-min-atomic(c)} $\Rightarrow$ \ref{it:prop:atomic-iff-min-atomic(a)} are obvious.
	
	\vskip 0.1cm
	\ref{it:prop:atomic-iff-min-atomic(a)} $\Rightarrow$ \ref{it:prop:atomic-iff-min-atomic(b)} We can assume that $H$ is not a group (i.e., $H \ne H^\times$), or else $\mathscr A(H)$ is empty and the claims are trivial. Since $H$ is atomic, we have (by definition) that 
	\[
	\emptyset \ne H \setminus H^\times \subseteq \mathrm{Sgp}\langle \mathscr A(H) \rangle_H.
	\]
	Then $\mathscr A(H)$ is non-empty, and this suffices, by Lemma \ref{lem:2.4}\ref{lem:2.4(i)}, to ensure that $\mathrm{Sgp}\langle \mathscr A(H) \rangle_H$ is contained in $H \setminus H^\times$. Putting everything together, we thus find that
	$
	H \setminus H^\times = \mathrm{Sgp}\langle \mathscr A(H) \rangle_H$.
	
	\vskip 0.1cm
	\ref{it:prop:atomic-iff-min-atomic(a)} $\Rightarrow$ \ref{it:prop:atomic-iff-min-atomic(c)} It is enough to consider that $\mathfrak a \preceq_H \mathfrak b$ and $\mathfrak b \preceq_H \mathfrak a$, for some $\mathscr A(H)$-words $\mathfrak a$ and $\mathfrak b$, if and only if $(\mathfrak a, \mathfrak b) \in \mathscr C_H$, see \cite[Propositions 4.2(iii) and 4.6(ii)]{An-Tr18} for further details.
\end{proof}

\begin{proposition}
\label{prop:divisor-closed-sub}
	Let $H$ be a monoid and $M$ a divisor-closed submonoid of $H$. The following hold: 
	\begin{enumerate}[label={\rm (\roman{*})}]
		\item\label{it:prop:divisor-closed-sub(i)}  $M^\times = H^\times$, $\mathscr{A}(M) = \mathscr{A}(H) \cap M$, 
		${\sf L}_M(x) = {\sf L}_H(x)$, $\mathcal{Z}_M(x) = \mathcal{Z}_H(x)$, and $\mathsf{Z}_M(x) = \mathsf{Z}_H(x)$.
		\item\label{it:prop:divisor-closed-sub(ii)} 
		${\sf L}_M^{\sf m}(x) = {\sf L}_H^{\sf m}(x)$, $\mathcal{Z}_M^{\sf m}(x) = \mathcal{Z}_H^{\sf m}(x)$, and $\mathsf{Z}_M^{\sf m}(x) = \mathsf{Z}_H^{\sf m}(x)$.
		\item\label{it:prop:divisor-closed-sub(iii)} $H$ is atomic \textup{(}resp., \textup{BF}, \textup{BmF}, \textup{FF}, \textup{FmF}, \textup{HF}, \textup{HmF}, factorial, or minimally factorial\textup{)} if and only if so is $\llangle x \rrangle_H$ for every $x \in H$.
	\end{enumerate}
\end{proposition}

\begin{proof}
	Parts \ref{it:prop:divisor-closed-sub(ii)} and  \ref{it:prop:divisor-closed-sub(iii)} are immediate from part \ref{it:prop:divisor-closed-sub(i)}, since we have already observed 
	that $\llangle X \rrangle_H$ is a divisor-closed sub\-monoid of $H$ for every $X \subseteq H$. As for part \ref{it:prop:divisor-closed-sub(i)}, see \cite[Proposition 2.21(ii)]{Fa-Tr18}.
\end{proof}
By and large, the present paper is about sufficient or necessary conditions for a monoid to be atomic, FF, BmF, etc. The literature abounds with results in this direction (with atoms often replaced by other ``elementary factors''), but --- apart from few exceptions in ``nearly cancellative'' settings --- these results are mostly about \emph{commutative} structures, see, e.g., \cite[Theorems 2.3, 2.4, and 3.1]{Co63}, \cite[Theorems 4 and 7]{Fl69}, \cite[Th\'eor\`eme de Structure]{Bouv74b}, \cite[Theorems 3.9, 3.11, 3.13, 4.4, and 4.9]{AnVL96}, \cite[Theorems 3.3, 3.4, and 3.6]{ChAnVLe11}, \cite[Proposition 3.1]{Sm13}, and \cite[parts (i) and (iv) of Theorem 2.28, and Corollary 2.29]{Fa-Tr18}.

\subsection{Presentations.} 
\label{sec:presentations}
In \S{ }\ref{sec:acyclic-monoids}, we will consider a couple of monoids defined via generators and relations in order to illustrate obstructions to a non-commutative analogue of fundamental results that hold true in the commutative setting. 
To this end, it is useful to recall a few basic facts about presentations, cf. \cite[\S{ }1.5]{Ho95}.

Let $X$ be a set and $R$ a binary relation on $\mathscr F(X)$. We denote by $R^\sharp$ the smallest monoid congruence on $\mathscr F(X)$ containing $R$. Formally, this means that
\[
R^\sharp := \bigcap \{\rho \subseteq \mathscr F(X) \times \mathscr F(X): \rho \text{ is a monoid congruence and } R \subseteq \rho\}.
\]
Accordingly, we have that $\mathfrak u \equiv \mathfrak v \bmod R^\sharp$ if and only if there are $\mathfrak z_0, \ldots, \mathfrak z_\ell \in \mathscr F(X)$, with $\mathfrak z_0 = \mathfrak u$ and $\mathfrak z_\ell = \mathfrak v$, such that for each $i \in \llb 0, \ell-1 \rrb$ there exist $\mathfrak p_i, \mathfrak q_i, \mathfrak q_i^\prime, \mathfrak r_i \in \mathscr F(X)$ with the property that 
\[
\mathfrak z_i = \mathfrak p_i \ast \mathfrak q_i \ast \mathfrak r_i, 
\quad
\mathfrak z_{i+1} = \mathfrak p_i \ast \mathfrak q_i^\prime \ast \mathfrak r_i,
\quad\text{and}\quad
\mathfrak q_i = \mathfrak q_i^\prime, \text{ or } \mathfrak q \RR \mathfrak q^\prime, \text{ or }\mathfrak q^\prime \RR \mathfrak q.
\]
In such a case, we call the finite sequence $\mathfrak z_0, \ldots, \mathfrak z_\ell$ an \emph{$R$-chain of length $\ell$ from $\mathfrak u$ to $\mathfrak v$} (notice that $\ell$ can also be zero, and if $\ell$ is actually zero then $\mathfrak u$ is necessarily equal to $\mathfrak v$).

With the above in place, we denote by $\langle X \mid R \rangle$ the monoid  obtained by taking the quotient of $\mathscr F(X)$ by the congruence $R^\sharp$.
We write $\langle X \mid R \rangle$ multiplicatively and call it a \emph{monoid presentation} (or simply a \emph{presentation}), with the elements of $X$ dubbed as the \emph{generators} 
and the pairs $(\mathfrak q, \mathfrak q^\prime) \in R$ as the \emph{defining relations} of the presentation. In particular, we 																																																																																																					refer to $\langle X \mid R \rangle$ as a \emph{finite presentation} if $X$ and $R$ are both finite. As is customary, we will usually identify an $X$-word $\mathfrak z$ with its equivalence class in $\langle X \mid R \rangle$ when there is no risk of confusion.

We take the \emph{left graph} of a monoid presentation $\langle X \mid R\rangle$ to be the undirected graph with vertex set $X$ and an edge from $y$ to $z$ for each pair $(y \ast \mathfrak y, z \ast \mathfrak z) \in R$ with $y, z \in X$ and
$\mathfrak y, \mathfrak z \in \mathscr F(X)$; note that this results in a loop when $y = z$, and in multiple (or parallel) edges between $y$ and $z$ if there are two or more defining relations of the form $(y \ast \mathfrak y, z \ast \mathfrak z)$.
The \emph{right graph} of a presentation is defined analogously, using the
right-most (instead of left-most) letters of each word from a defining relation. 

A monoid is \emph{Adian} if it is isomorphic to a finite presentation 
whose left and right graphs are \emph{cycle-free}, that is, do not contain cycles (including loops). Our interest for Adian monoids stems from the following result, commonly referred to as Adian's embedding theorem since first proved in \cite[Theorem II.4]{Ad66}.

\begin{theorem}
	\label{th:adian-theorem}
	Every Adian monoid embeds into a group
	\textup{(}and hence is cancellative\textup{)}.
\end{theorem}
We shall use Theorem \ref{th:adian-theorem} in Examples \ref{exa:non-atomic-2-generator-1-relator-canc-mon} and \ref{exa:5.4} to show that certain presentations are cancellative. But a more ``hands-on approach'' will be necessary to prove the cancellativity of another presentation (with four generators and countably infinite many defining relations) we consider in Example \ref{exa:atomic-2-gen-1-rel-canc-monoid}. 

\section{Finitely generated monoids et similia}
\label{sec:lfgu-monoids}
Commutative monoids that are finitely generated after modding out their group of units, play a central role in the classical theory of factorization as developed in \cite{GeHK06}. In the present section, we consider a non-commutative generalization of such structures and show how this results in a class of monoids with remarkable arithmetic properties.

\begin{definition}\label{def:fg-monoids-and-the-like}
	Let $H$ be a monoid. We say that $H$ is
	\begin{itemize}
		\item \emph{finitely generated} (shortly, f.g.) if $H = \{1_H\} \cup \mathrm{Sgp}\langle A \rangle_H$ for some finite set $A \subseteq H$;
		\item \emph{locally finitely generated} (shortly, l.f.g.) if $\llangle x \rrangle_H$ is f.g. for every $x \in H$;
		\item \emph{finitely generated up to units} (shortly, f.g.u.) if $H \setminus H^\times \subseteq \mathrm{Sgp}\langle H^\times A H^\times \rangle_H$ for a finite $A \subseteq H$;
		\item \emph{locally finitely generated up to units} (shortly, l.f.g.u.) if $\llangle x \rrangle_H$ is f.g.u. for every $x \in H$.
	\end{itemize}
\end{definition}
It is straightforward that a commutative monoid $H$ is f.g.u. (resp., l.f.g.u.) if and only if the quotient $H / H^\times$ is f.g. (resp., l.f.g.), and it is obvious that f.g. (resp., l.f.g.) monoids are f.g.u. (resp., l.f.g.u.). 

More interestingly, cancellative l.f.g.u. commutative monoids are \textup{FF} by \cite[Proposition 2.7.8.4]{GeHK06}, and unit-cancellative l.f.g.u. commutative monoids are BF by \cite[Proposition 3.4]{FGKT}. 
All together, some of the main contributions of this paper (viz., Theorem \ref{th:lfgu-BF-is-FF}\ref{it:th:lfgu-BF-is-FF(i)} and Corollaries \ref{cor:lfgu-is-FmF}, \ref{cor:lfgu-acyclic-mons-are-atomic}, and \ref{cor:comm-unit-canc-lfgu-mons-are-FF}) will provide an overarching generalization of these results to a non-commutative setting. 

We start with a couple of technical lemmas, the first of which is perhaps of independent interest.

\begin{lemma}\label{lem:minimal-generating-set}
Let $H$ be a monoid, and let $A, Q \subseteq H$ such that $Q^2 \subseteq Q$. The following hold:
	\begin{enumerate}[label=\textup{(\roman{*})}]
		\item\label{lem:minimal-generating-set(i)}
		If $A$ is finite, then there exists a subset $\bar{A}$ of $A$ with the property that $\mathrm{Sgp} \langle QAQ \rangle_H = \mathrm{Sgp} \langle Q \bar{A} Q \rangle_H$ and $a \notin \mathrm{Sgp}\langle Q\bar{A}Q \setminus QaQ \rangle_H$ for every $a \in \bar{A}$.
		\item\label{lem:minimal-generating-set(ii)} 
		If $1_H \in Q$, then $\mathrm{Sgp}\langle QAQ \rangle_H \subseteq Q \cup \mathrm{Sgp}\langle Q (A \setminus Q) Q \rangle_H$.
	\end{enumerate}
\end{lemma}

\begin{proof}
	\ref{lem:minimal-generating-set(i)} Assume that $\kappa := |A| < \infty$. If $\kappa = 0$ or $a \notin \mathrm{Sgp}\langle QAQ \setminus QaQ \rangle_H$ for every $a \in A$, the conclusion is trivial (in both cases, just take $\bar A := A$). Otherwise, let $a \in A$ such that 
	\begin{equation}\label{equ:pick-one-out}
	a \in \mathrm{Sgp}\langle QAQ \setminus QaQ \rangle_H,
	\end{equation}
	and define $A^\prime := A \setminus \{a\}$. 
	We claim that 
	$$
	\mathrm{Sgp}\langle QAQ \rangle_H \subseteq \mathrm{Sgp}\langle QA^\prime Q \rangle_H.
	$$
	Since $|A^\prime| < \kappa$ and $\mathrm{Sgp}\langle QA^\prime Q \rangle_H \subseteq \mathrm{Sgp} \langle QAQ \rangle_H$, this will finish the proof (by induction on $\kappa$).
	
    To demonstrate the claim, let $x \in \mathrm{Sgp} \langle QAQ \rangle_H$. By definition, this means that
	\begin{equation}\label{equ:QAQ-factorization}
	x = u_1 a_1 v_1 \cdots u_n a_n v_n, \quad
	\text{for some }
	u_1, \ldots, u_n, v_1, \ldots, v_n \in Q
	\text{ and }
	a_1, \ldots, a_n \in A.
	\end{equation}
	Set $I := \{i \in \llb 1, n \rrb: a_i = a\}$. We have to show that $x \in \mathrm{Sgp}\langle QA^\prime Q \rangle_H$. So, fix $i \in \llb 1, n \rrb$. 
	
	By \eqref{equ:QAQ-factorization}, it is sufficient to prove that $u_i a_i v_i \in \mathrm{Sgp}\langle QA^\prime Q \rangle_H$. If $i \notin I$, then $a_i \in A^\prime$ and we are done. If not, $a_i = a$ and, by \eqref{equ:pick-one-out}, 
	there are $\bar{u}_1, \ldots, \bar{u}_m, \bar{v}_1, \ldots, \bar{v}_m \in Q$ and $\bar{a}_1, \ldots, \bar{a}_m \in A^\prime$ such that
	$$
	a_i = \bar{u}_1 \bar{a}_1 \bar{v}_1 \cdots \bar{u}_m \bar{a}_m \bar{v}_m.
	$$
	It follows that $u_i a_i v_i = \hat{u}_1 \bar{a}_1 \hat{v}_1 \cdots \hat{u}_m \bar{a}_m \hat{v}_m$, where
	$$
	\hat{u}_j := 
	\left\{
	\begin{array}{ll}
	\! u_i \bar{u}_1  & \text{if }j = 1  \\ 
	\! \bar{u}_j  & 1 < j \le m
	\end{array}
	\right.
	\quad\text{and}\quad
	v_j :=
	\left\{
	\begin{array}{ll}
	\! \bar{v}_j  & \text{if }1 \le j < m  \\ 
	\! \bar{v}_m v_i  & j = m
	\end{array}
	\right.;
	$$
	and since $Q^2 \subseteq Q$ (by hypothesis), $\hat{u}_j, \hat{v}_j \in Q$ for all $j \in \llb 1, m \rrb$. Then $u_i a_i v_i \in \mathrm{Sgp} \langle QA^\prime Q \rangle_H$, as wished.
	
	\vskip 0.1cm
	\ref{lem:minimal-generating-set(ii)} Assume $1_H \in Q$, and suppose that $x = u_1 a_1 v_1 \cdots u_n a_n v_n$ for some $u_1, \ldots, u_n, v_1, \ldots, v_n \in Q$ and $a_1, \ldots, a_n \in A$. We have to prove that $x \in Q \cup \mathrm{Sgp}\langle Q(A \setminus Q)Q \rangle_H$.
	
	If $a_1, \ldots, a_n \in Q$ or $a_1, \ldots, a_n \in A \setminus Q$, then we are done (note that $Q^k \subseteq Q$ for all $k \in \mathbf N^+$, by the fact that $Q^2 \subseteq Q$). Otherwise, $n \ge 2$ and there exists $\hat\imath \in \llb 1, n \rrb$ with $a_{\hat\imath} \in Q$. Accordingly, we can write
	$$
	x = \bar{u}_1 \bar{a}_1 \bar{v}_1 \cdots \bar{u}_{n-1} \bar{a}_{n-1} \bar{v}_{n-1}, 
	$$
	where we take $\bar{a}_i := a_i$ for $1 \le i < \hat\imath$ and $\bar{a}_i := a_{i+1}$ for $\hat\imath \le i \le n-1$, and we define
	$$
	\bar u_i := 
	\left\{
	\begin{array}{ll}
	\! u_i & \text{if } 1 \le i < \hat\imath \\
	\! u_i a_i v_i u_{i+1} & \text{if } i = \hat\imath \ne n \\
	\! u_{i+1} & \text{if }\hat\imath < i \le n-1
	\end{array}
	\right.\!
	\quad\text{and}\quad
	\bar v_i := 
	\left\{
	\begin{array}{ll}
	\! v_i & \text{if } 1 \le i < \hat\imath \le n-1 \\
	\! v_{i+1} & \text{if }\hat\imath \le i \le n-1 \\
	\! v_i u_{i+1} a_{i+1} v_{i+1} & \text{if } i+1 = \hat\imath = n
	\end{array}
	\right.\!.
	$$
	It follows, by induction on $n$, that $x \in H$ (as was desired), since it is clear from the above that $\bar u_i, \bar v_i \in Q$ and $\bar a_i \in A$ for every $i \in \llb 1, n-1 \rrb$.
\end{proof}

\begin{lemma}
\label{lem:atoms-of-atomic-fgu-mon}
	Let $H$ be an f.g.u. monoid. There exists a finite set $\bar{A} \subseteq \mathscr A(H)$ such that $\mathscr A(H) = H^\times \bar{A} H^\times$ and $a \not \simeq_H b$ for all $a, b \in \bar{A}$ with $a \ne b$.
\end{lemma}

\begin{proof}
	We may suppose that $\mathscr A(H)$ is non-empty, or else the claim is obvious (just take $\bar{A} := \emptyset$). 
	
	Since $H$ is an f.g.u. monoid, $H \setminus H^\times \subseteq \mathrm{Sgp}\langle H^\times A H^\times \rangle_H$ for some finite $A \subseteq H$; in particular, we can assume, by Lemma \ref{lem:minimal-generating-set}\ref{lem:minimal-generating-set(ii)}, that $A \subseteq H \setminus H^\times$ and, hence, $H^\times A H^\times \subseteq H \setminus H^\times$. It follows that
	\begin{equation}\label{equ:double-containment}
	\emptyset \ne \mathscr A(H) \subseteq H \setminus H^\times = \mathrm{Sgp}\langle H^\times A H^\times \rangle_H,
	\end{equation}
 	because $\mathscr A(H)$ being non-empty guarantees, by Lemma \ref{lem:2.4}\ref{lem:2.4(ii)}, that 
 	\[
 	\mathrm{Sgp}\langle H^\times A H^\times \rangle_H \subseteq \textrm{Sgp} \langle H \setminus H^\times \rangle_H \subseteq H \setminus H^\times.
 	\]
 	Set $A^\prime := A \cap \mathscr A(H)$. We know from Lemma \ref{lem:2.4}\ref{lem:2.4(i)} that
 	\begin{equation*}
 	uav \in \mathscr A(H),
 	\quad\text{for all } u, v \in H^\times
 	\text{ and }
 	a \in \mathscr A(H). 
 	\end{equation*}
 	Consequently, it is evident that
	\begin{equation}\label{equ:almost-equality}
	H^\times A^\prime H^\times \subseteq H^\times \mathscr A(H) H^\times = \mathscr A(H).
	\end{equation}
	We claim that the inclusion in the previous display can be reversed, leading to
	\begin{equation}\label{equ:decompositional-identity}
	\mathscr A(H) = H^\times A^\prime H^\times.
	\end{equation}
	Indeed, fix $a \in \mathscr A(H)$.
  	From \eqref{equ:double-containment}, we have that $a = u_1 a_1 v_1 \cdots u_n a_n v_n$ for some $u_1, \ldots, u_n, v_1, \ldots, v_n \in H^\times$ and $a_1, \ldots, a_n \in A$. However, this is only possible if $n = 1$, because $u_i a_i v_i \in H \setminus H^\times$ for every $i \in \llb 1, n \rrb$ (by definition, an atom can not be a product of two non-units, and we have already established that, in $H$, a product of non-units is still a non-unit). We therefore conclude from \eqref{equ:almost-equality} that $a_1 = u_1^{-1} a v_1^{-1} \in \mathscr A(H)$ and, hence, $a \in H^\times A^\prime H^\times \subseteq \mathscr A(H)$. The latter is enough to prove the claim, since $a$ was arbitrary.
	
	Now, considering that $H^\times$ is a subgroup of $H$ and $A^\prime$ is a finite subset of $\mathscr A(H)$, we obtain from \eqref{equ:decompositional-identity} and Lemma \ref{lem:minimal-generating-set}\ref{lem:minimal-generating-set(i)} that there is a non-empty finite set $A^{\prime\prime}$, consisting of atoms of $H$, such that 
	\begin{equation} \label{equ:another-identity}
	\mathscr A(H) \subseteq \mathrm{Sgp} \langle H^\times A^\prime H^\times \rangle_H = \mathrm{Sgp} \langle H^\times A^{\prime\prime} H^\times \rangle_H
	\end{equation}
	and 
	\begin{equation}\label{equ:generating-minimality}
	a \notin \mathrm{Sgp}\langle H^\times A^{\prime\prime} H^\times \setminus H^\times a H^\times \rangle_H,
	\quad\text{for every }a \in A^{\prime\prime}.
	\end{equation}
	So, choosing one element from each set in the finite family $\{H^\times a H^\times: a \in A^{\prime\prime}\}$ and noting that $\simeq_H$ is an equivalence on $H$ (with the result that $H^\times x H^\times \cap H^\times y H^\times$ is non-empty, for some $x, y \in H$, if and only if $H^\times x H^\times = H^\times y H^\times$), it is straightforward from \eqref{equ:another-identity} and \eqref{equ:generating-minimality} that there is a finite set $\bar{A} \subseteq A^{\prime\prime} \subseteq \mathscr A(H)$ with the property that $\mathscr A(H) \subseteq \mathrm{Sgp}\langle H^\times \bar{A} H^\times \rangle_H$ and $a \not\simeq_H b$ for all $a, b \in \bar{A}$ with $a \ne b$. 
	
	Similarly as in the derivation of \eqref{equ:decompositional-identity}, this implies that $\mathscr A(H) \subseteq H^\times \bar{A} H^\times$, thus completing the proof, because it is clear from \eqref{equ:almost-equality} and the above that $H^\times \bar{A} H^\times \subseteq \mathscr A(H)$.
\end{proof}

\begin{theorem}\label{th:lfgu-BF-is-FF}
	Let $H$ be an l.f.g.u. monoid. The following hold:
	\begin{enumerate}[label=\textup{(\roman{*})}]
	\item\label{it:th:lfgu-BF-is-FF(i)} $H$ is \textup{BF} if and only if it is \textup{FF}.
	\item\label{it:th:lfgu-BF-is-FF(ii)} $\mathsf Z_H^{\sf m}(x)$ is finite for every $x \in H$.
	\end{enumerate}
\end{theorem}

\begin{proof}
	To start with, we note that, thanks to Proposition \ref{prop:divisor-closed-sub}\ref{it:prop:divisor-closed-sub(iii)}, one can assume without loss of generality that $H$ is an f.g.u. monoid. 
	Accordingly, we derive from Lemma \ref{lem:atoms-of-atomic-fgu-mon} that there is a finite set $\bar{A} \subseteq \mathscr A(H)$ such that $
	\mathscr A(H) = H^\times \bar{A} H^\times$ and $a \not\simeq_H b$ for all $a, b \in \bar{A}$ with $a \ne b$. 
	
	We thus have a well-defined map 
	$
	\varphi: \mathscr F(\mathscr A(H))/\mathscr C_H \to \mathscr F(\bar A)$
	sending the congruence class of $\varepsilon_{\mathscr A(H)}$ (relative to $\mathscr C_H$) to $\varepsilon_{\bar A}$, and the congruence class of a non-empty $\mathscr A(H)$-word $\mathfrak a = a_1 \ast \cdots \ast a_n$ of length $n$ to the unique $\bar A$-word $\bar a_1 \ast \cdots \ast \bar a_n$ of length $n$ such that $a_i \simeq_H \bar a_i$ for each $i \in \llb 1, n \rrb$.
	
	Moreover, we see that, for each $k \in \mathbf N$, there are at most $|\bar A|^k$ words $\mathfrak a \in \mathscr F(\mathscr A(H))$ of length $k$ that are pairwise incongruent modulo $\mathscr C_H$, because $
	\mathscr A(H) = H^\times \bar{A} H^\times$ and hence $\mathfrak a = u_1a_1v_1 \ast \cdots \ast u_k a_k v_k$ for some $u_1, \ldots, u_k, v_1, \ldots, v_k \in H^\times$ and $a_1, \ldots, a_k \in \bar{A}$. In other words, we have
	\begin{equation}\label{equ:(9)}
	\bigl|\bigl\{ \llb \mathfrak a \rrb_{\mathscr C_H}: \mathfrak a \in \mathscr F(\mathscr A(H)) \text{ and } \|\mathfrak a\|_{\mathscr A(H)} \le k \bigr\}\bigr| \le \sum_{i=0}^k |\bar A|^i < \infty, \quad \text{for every }k \in \mathbf N.
	\end{equation}
	With these premises in place, it is now not difficult to finish the proof of the theorem.
	
	\vskip 0.1cm
	\ref{it:th:lfgu-BF-is-FF(i)} Let $H$ be a BF-monoid with $H \ne H^\times$ (it is fairly obvious that an FF-monoid is BF and a group is an FF-monoid), and let $x \in H \setminus H^\times$. Then $\mathsf L_H(x) \subseteq \llb 1, n \rrb$ for some $n \in \mathbf N^+$, implying, by \eqref{equ:(9)}, that
	\[
	|\mathsf Z_H(x)| \le \bigl|\bigl\{ \llb \mathfrak a \rrb_{\mathscr C_H}: \mathfrak a \in \mathscr F(\mathscr A(H)) \text{ and } \|\mathfrak a\|_{\mathscr A(H)} \le n \bigr\} \bigr| < \infty. 
	\]
	This suffices to proves that $H$ is an FF-monoid (since $x$ was an arbitrary element in $H \setminus H^\times$).
	
	\vskip 0.1cm
	\ref{it:th:lfgu-BF-is-FF(ii)} Suppose to the contrary that $\mathsf Z_H^{\sf m}(x)$ is infinite for some $x \in H$. It then follows from \eqref{equ:(9)} that there is a sequence $(\mathfrak a_i)_{i \ge 1}$ of $\preceq_H$-minimal factorizations of $x$ such that 
	\[
	1 \le \|\mathfrak a_i\|_{\mathscr A(H)} < \|\mathfrak a_{i+1}\|_{\mathscr A(H)}, 
	\quad\text{for every } i \in \mathbf N^+.
	\] 
	Set $\bar{\mathfrak a}_1 := \varphi(\mathfrak a_1)$, $\bar{\mathfrak a}_2 := \varphi(\mathfrak a_2)$, \ldots, where $\varphi$ is the function defined in the premises of the proof. Then by Higman's lemma (that is, Lemma \ref{lem:higman-lemma}), there are $\hat{\imath}, \hat{\jmath} \in \mathbf N^+$ with $\hat{\imath} < \hat{\jmath}$ such that $\bar{\mathfrak a}_{\hat{\imath}}$ is a non-empty proper subword of $\bar{\mathfrak a}_{\hat{\jmath}}$ (by construction, $\bar A$ is a finite alphabet and $\bar{\mathfrak a}_1, \bar{\mathfrak a}_2, \ldots$ are $\bar A$-words). 
	
	Let $m$ denote the length of $\bar{\mathfrak a}_{\hat{\imath}}$ and $n$ the length of $\bar{\mathfrak{a}}_{\hat{\jmath}}$. By the above, there exist $a_1, \ldots, a_n \in \bar{A}$ and a strictly increasing function $\sigma: \llb 1, m \rrb \to \llb 1, n \rrb$ such that $\bar{\mathfrak a}_{\hat{\imath}} = a_{\sigma(1)} \ast \cdots \ast a_{\sigma(m)}$ and $\bar{\mathfrak a}_{\hat{\jmath}} = a_1 \ast \cdots \ast a_n$. On the other hand, we have by definition of $\varphi$ that, for some $s_1, t_1, \ldots, s_m, t_m, u_1, v_1, \ldots, u_n, v_n \in H^\times$,
	\begin{equation*}
	\mathfrak a_{\hat\imath} = s_1 a_{\sigma(1)} t_1 \ast \cdots \ast s_m a_{\sigma(m)} t_m
	\quad\text{and}\quad
	\mathfrak a_{\hat\jmath} = u_1 a_1 v_1 \ast \cdots \ast u_n a_n v_n.
	\end{equation*}
	However, since $\pi_H(\mathfrak a_{\hat\imath}) = \pi_H(\mathfrak a_{\hat\jmath}) = x$, $\|\mathfrak a_{\hat\imath}\|_{\mathscr A(H)} < \|\mathfrak a_{\hat\jmath}\|_{\mathscr A(H)}$, and $s_i a_{\sigma(i)} t_i \simeq_H u_{\sigma(i)} a_{\sigma(i)} v_{\sigma(i)}$ for each $i \in \llb 1, m \rrb$, this yields $\mathfrak a_{\hat\imath} \prec_H \mathfrak a_{\hat\jmath}$, contradicting the $\preceq_H$-minimality of $\mathfrak a_{\hat\jmath}$.
\end{proof}

\begin{corollary}\label{cor:lfgu-is-FmF}
Every atomic l.f.g.u. monoid is \textup{FmF}.
\end{corollary}

\begin{proof}
This is a direct consequence of Proposition \ref{prop:atomic-iff-min-atomic} and Theorem \ref{th:lfgu-BF-is-FF}\ref{it:th:lfgu-BF-is-FF(ii)}.\end{proof}

Of course, an l.f.g.u. monoid $H$ need not be atomic (e.g., due to the presence of non-trivial idempotents that do not factor into a product of atoms). Less trivially, there exist cancellative f.g. monoids that are atomic but not BF, as will be seen in the next section (Example \ref{exa:atomic-2-gen-1-rel-canc-monoid}).
All in all, this shows that Theorem \ref{th:lfgu-BF-is-FF}\ref{it:th:lfgu-BF-is-FF(i)} and Corollary \ref{cor:lfgu-is-FmF} are, in a sense, best possible (see also Corollary \ref{cor:lfgu-acyclic-mons-are-atomic}).

\section{Acyclic monoids} 
\label{sec:acyclic-monoids}

We know from Corollary \ref{cor:lfgu-is-FmF} that an atomic l.f.g.u. monoid is FmF, and in the current section we look for sufficient conditions for an l.f.g.u. monoid to be atomic (or more).

\begin{definition}\label{def:acyclic}
	We call a monoid $H$ \emph{acyclic} if $x \ne uxv$ for all $u, v, x \in H$ with $u \notin H^\times$ or $v \notin H^\times$.
\end{definition}

We will see that acyclic monoids ``abound in nature''. But first a remark of general character:

\begin{remark}\label{rem:history} 
	Obviously, a cancellative or acyclic monoid is unit-cancellative; a free monoid is acyclic; and a commutative monoid is acyclic if and only if it is unit-cancellative. However, a non-commutative cancellative monoid need not be acyclic (Example \ref{exa:non-atomic-2-generator-1-relator-canc-mon}); and there is a wide assortment of commutative unit-cancellative monoids that are non-cancellative, including families of monoids of ideals, of monoids of modules, and of power monoids, see \cite[\S\S{ }3.2--3.4]{FGKT} and \cite[\S{ }3]{Fa-Tr18}. 
\end{remark}

With this said, we discuss in some detail a few basic features of three ``large'' families of non-commutative acyclic monoids arising from the literature. In particular, the first example is inspired by J.F. Ritt's seminal work \cite{Ri22} on univariate polynomials that are irreducible with respect to functional composition (see \cite{KrTi15} for a survey on this subject); and the second by the work of various authors on the arithmetic of matrix rings (see \cite{BaPo11,BaBaGo14,BaJe16}, \cite[\S{ }5]{Sm16}, and the bibliography therein). 

We shall refer to \cite[Chap. I, \S{ }8]{Bou98} for fundamental aspects of ring theory, and to \cite[Chap. II, \S{ }10 and Chap. III, \S{ }8.3]{Bou98} for fundamental aspects of linear algebra over commutative rings.

\begin{example}
\label{exa:polynomial-transformation-monoid}
Let $0_R$ be the zero, $1_R$ the multiplicative identity, $R^\times$ the group of (multiplicative) units, and $R^\bullet$ the set of regular (or cancellable) elements of a non-trivial commutative ring $R$, and let $R[X]$ be the ring of polynomials in one indeterminate $X$ with coefficients in $R$. (We recall that an element $a \in R$ is regular if $ax \ne 0_R$ for every non-zero $x \in R$, and $R$ is non-trivial if $0_R \ne 1_R$.)

It is obvious (and well known) that $R^\bullet$ is a submonoid of the multiplicative monoid of $R$ and $R^\times$ is a subgroup of $R^\bullet$. Hence, one can readily verify that the set of all non-constant polynomials $\sum_{i=0}^n \alpha_n X^n \in R[X]$ whose leading coefficient $\alpha_n$ belongs to $R^\bullet$, is itself a monoid, herein denoted by $R_\circ[X]$, under the operation $\circ$ of \emph{functional composition} defined by:
\[
f \circ g(X) := f(g(X)), \quad \text{for all } f, g \in R_\circ[X].
\]
The identity of $R_\circ[X]$ is the polynomial $\mathbf{1}_R(X) := X \in R[X]$, and it is  easily seen that 
\begin{equation}\label{equ:4.9}
R_\circ[X]^\times = \{\alpha X + \beta \in R[X]: \alpha \in R^\times \text{ and }\beta \in R\}.
\end{equation}
We claim that $R_\circ[X]$ is an acyclic monoid. 
Indeed, assume $f = u \circ f \circ v$ for some $f, u, v \in R_\circ[X]$. Then $\deg(f) = \deg (u) \, \deg (f) \, \deg (v) \ne 0$, which is only possible if $\deg (u) = \deg (v) = 1$, that is, $u(X) = aX + b$ and $v(X) = cX + d$ for some $a, c \in R^\bullet$ and $b, d \in R$. It follows that
\begin{equation*}
q X^{\deg(f)} + g(X) = f(X) = u \circ f \circ v(X) = a f(cX + d) + b = a q c X^{\deg (f)} + h(X), 
\end{equation*}
where $q$ is the leading coefficient of $f$, and $h$ and $g$ are  polynomials in $R[X]$ of degree $\le \deg(f)-1$. This in turn implies $q (ac - 1_R) = 0_R$, with the result that $ac = 1_R$ because $q$ is a regular element of $R$. Then $a, c \in R^\times$, and we conclude, by \eqref{equ:4.9}, that $u, v \in R_\circ[X]^\times$ (as wished).

Finally, let $\kappa$ be the order of $1_R$ relative to the additive group of $R$, and consider the polynomials 
\[
F(X) = X^2+X \in R_\circ[X], 
\quad
G(X) = X - 1_R \in R_\circ[X],
\quad\text{and}\quad
H(X) = X^2 \in R_\circ[X]. 
\]
A simple calculation reveals that $F \circ G(X) = X^2-X \ne X^2 + X - 1_R = G \circ F(X)$. Moreover, we find that, if $\kappa = 2$, then $F \circ G = F \circ \mathbf{1}_R$ although $G \ne \mathbf{1}_R$;  and if $\kappa \ne 2$, then $H \circ \textbf{1}_R = H \circ (-\textbf{1}_R)$ although $\textbf{1}_R \ne -\textbf{1}_R$. On the whole, this shows that $R_\circ[X]$ is neither commutative nor cancellative (regardless of the actual choice of the ring $R$).
\end{example}

\begin{example}\label{exa:regular-matrix}
Fix an integer $n \ge 2$ and let $\mathcal{M}_n(R)$ be the ring of all $n$-by-$n$ matrices with entries in a non-trivial, commutative ring $R$ (endowed with the usual operations of entrywise addition and row-by-column multiplication). By \cite[Chap. III, \S{ }8.3, Proposition 5]{Bou98}, a matrix in $\mathcal M_n(R)$ is a unit (with respect to multiplication) if and only if its determinant is a unit of $R$ (we keep using notation and terminology introduced in Example \ref{exa:polynomial-transformation-monoid}); moreover, we have
\begin{equation}\label{equ:cauchy-binet}
\det(AB) = \det(A) \det(B), 
\quad\text{for all }
A, B \in \mathcal M_n(R).
\end{equation}
Since $R^\bullet$ is a sub\-monoid of the multiplicative monoid of $R$ and $R^\times$ is a subgroup of $R^\bullet$, it follows that 
\[
\mathcal{S}_n(R) := \{A \in \mathcal{M}_n(R): \det(A) \in R^\bullet\}
\]
is a submonoid of the multiplicative monoid of $\mathcal{M}_n(R)$ whose group of units is precisely the general linear group of degree $n$ over $R$, that is to say,
\begin{equation}\label{equ:cramer}
\mathcal{S}_n(R)^\times = \text{GL}_n(R) := \{A \in \mathcal M_n(R): \det(A) \in R^\times\}.
\end{equation}
We will find that $\mathcal{S}_n(R)$ is cancellative, acyclic, and non-commutative (for any choice of the ring $R$).

First, assume $ABC = B$ for some $A, B, C \in \mathcal{S}_n(R)$. Then \eqref{equ:cauchy-binet} yields $\det(B) = \det(A) \det(B) \det(C)$, 
and hence $(\det(A) \det(C) - 1_R) \det(B) = 0_R$. But this can only happen if $\det(A) \det(C) = 1_R$, because $\det(B) \in R^\bullet$. Therefore, $\det(A)$ and $\det(C)$ are units of $R$, implying by \eqref{equ:cramer} that $\mathcal{S}_n(R)$ is acyclic.

On the other hand, it is an easy exercise to see that $\mathcal{S}_n(R)$ is non-commutative. Indeed, consider the matrix $D \in \mathcal M_n(R)$ whose $(i,j)$-entry is $1_R$ if $i = j$ or $i = j - 1 = 1$, and $0_R$ otherwise ($1 \le i, j \le n$). Then $D$ and its transpose $D^t$ are both in $\mathcal{S}_n(R)$, since their determinant is $1_R$; but $D D^t \ne D^t D$ because the $(1,1)$-entry of $D D^t$ is $1_R + 1_R$ and the $(1,1)$-entry of $D^t D$ is $1_R$.

Lastly, denote by $\mathcal Q(R)$ the total ring of fractions of $R$ (see \cite[Chap. I, \S{ }8.12]{Bou98} for terminology), by $\mathcal Q(R)^\times$ the group of (multiplicative) units of $\mathcal Q(R)$, and
by $\mathcal M_n(\mathcal Q(R))$ the ring of $n$-by-$n$ matrices with entries in $\mathcal Q(R)$.
Then
$R^\bullet \subseteq \mathcal Q(R)^\times$ (essentially by definition), and this in turn entails that the inclusion map $\mathcal M_n(R) \hookrightarrow \mathcal M_n(\mathcal Q(R))$ yields an injective monoid homomorphism from $\mathcal S_n(R)$ to $\text{GL}_n(\mathcal Q(R))$. Thus $\mathcal S_n(R)$ is cancellative, since it embeds (as a monoid) into a group.
\end{example}

\begin{example}\label{exa:4.5}
We say that a monoid $H$ is \emph{normalizing} if $aH = Ha$ for every $a \in H$. Obviously, every commutative monoid is normalizing; and important examples of non-commutative normalizing monoids arise from the study of ring extensions and monoid algebras \cite{Je-Ok}.

Let $H$ be a unit-cancellative normalizing monoid, and suppose $uxv = x$ for some $u, v, x \in H$. Since $
H$ is normalizing, there then exist $\bar{u}, \bar{v} \in H$ such that $x = uxv = x\bar{u} v = u\bar{v}x$; and since $H$ is unit-cancellative, this is only possible if $\bar{u} v$ and $u\bar{v}$ are units. It thus follows by Lemma \ref{lem:2.4}\ref{lem:2.4(ii)} that $u$ and $v$ are both units; whence we conclude that $H$ is acyclic.
\end{example}

Next, we establish that BF-monoids are acyclic (Proposition \ref{prop:BF-implies-acyclic}) and show how to build new acyclic monoids from old ones (Proposition \ref{prop:new-from-old}). But first, we need to fix a mistake in \cite[Remark 2.4]{Fa-Tr18}.

\begin{proposition}\label{prop:subadditivity-of-length-sets}
	Let $H$ be a monoid, and let $x, y \in H$. The following hold:
	\begin{enumerate}[label=\textup{(\roman{*})}]
		\item\label{it:prop:subadditivity-of-length-sets(i)} $\mathsf L_H(x) + \mathsf L_H(y) \subseteq \mathsf L_H(xy)$.
		\item\label{it:prop:subadditivity-of-length-sets(ii)} If $\mathsf L_H(x)$ and $\mathsf L_H(y)$ are both non-empty or at least one of $x$ and $y$ is a unit, then 
		\begin{equation}\label{equ:sup-inequality}
		\sup \mathsf L_H(x) + \sup \mathsf L_H(y) \le \sup \mathsf L_H(xy), 
		\end{equation}
		where $\sup \emptyset := 0$. In particular, \eqref{equ:sup-inequality} holds when $H$ is atomic.
	\end{enumerate}
\end{proposition}

\begin{proof}
	Part \ref{it:prop:subadditivity-of-length-sets(i)} is a direct consequence of \cite[Example 2.2]{Tr19a}. As for part \ref{it:prop:subadditivity-of-length-sets(ii)}, it follows from \ref{it:prop:subadditivity-of-length-sets(i)} that
	\[
	\sup (\mathsf L_H(x) + \mathsf L_H(y)) \le \sup \mathsf L_H(xy).
	\]
	So, since $\sup(X + Y) = \sup X + \sup Y$ for all non-empty subsets $X$ and $Y$ of $\bf R$, we see that \eqref{equ:sup-inequality} holds if $\mathsf L_H(x)$ and $\mathsf L_H(y)$ are both non-empty. Accordingly, assume from this point on that $x$ is a unit (the case when $y \in H^\times$ is similar). We have by Lemma \ref{lem:2.4}\ref{lem:2.4(ii)} and \cite[Lemma 2.2(iv)]{Fa-Tr18} that
	\[
	\mathsf L_H(u) \subseteq \{0\}
	\quad\text{and}\quad
	\mathsf L_H(z) = \mathsf L_H(uzv), 
	\quad\text{for all } u, v \in H^\times \text{ and }z \in H \setminus H^\times. 
	\]
	Since $x$ is a unit and $H^\times$ is a subgroup of $H$, it follows that $\sup \mathsf L_H(x) = 0$ and $\sup \mathsf L_H(y) = \sup \mathsf L_H(xy)$, whence $\sup \mathsf L_H(xy) =
	\sup \mathsf L_H(x) + \sup \mathsf L_H(y)$ and we are done.
\end{proof}

\begin{proposition}
\label{prop:BF-implies-acyclic}
	Every \textup{BF}-monoid is acyclic.
\end{proposition}

\begin{proof}
	Let $H$ be a BF-monoid, and suppose for the sake of contradiction that $H$ is not acyclic, namely, there exist $u, v, x \in H$ with $u \notin H^\times$ or $v \notin H^\times$ such that
	$uxv = x$. In fact, we may assume that $u \notin H^\times$ (the other case is essentially the same). Then $u$ is a non-empty product of atoms (because $H$ is atomic), and hence $\mathsf L_H(u)$ is a non-empty subset of $\mathbf N^+$. 
	On the other hand, $x = uxv$ yields (by induction) that $x = u^k x v^k$ for every $k \in \mathbf N$. So putting it all together, we conclude from Proposition \ref{prop:subadditivity-of-length-sets}\ref{it:prop:subadditivity-of-length-sets(ii)} that
	\[
	\sup \mathsf L_H(x) = \sup \mathsf L_H(u^k x v^k) \geq \sup \mathsf L_H(u^k)
	\ge k\, \sup \mathsf L_H(u) \ge k, \quad \text{for each } k \in \mathbf N.
	\]
	This, however, means that $\sup \mathsf L_H(x) = \infty$, contradicting the hypothesis that $H$ is a BF-monoid.
\end{proof}

\begin{example}
In the notation of Example \ref{exa:polynomial-transformation-monoid}, assume that $R$ is a field. Then we have by \eqref{equ:4.9} that every polynomial $f \in R_\circ[X]$ of degree $1$ is a unit. In consequence, we find that the function
\[
\lambda: R_\circ[X] \to \mathbf N: f \mapsto \deg(f)
\]	
is a \emph{length function} on $R_\circ[X]$, meaning that, if $f, g \in R_\circ[X]$ and $g = u \circ f \circ v$ for some $u, v \in R_\circ[X]$ such that $u$ or $v$ is not a unit, then $\lambda(g) < \lambda(f)$, see \cite[Definition 2.26]{Fa-Tr18}. In fact, it is clear that
\[
\lambda(f \circ g) = \lambda(f)\, \lambda(g), \quad\text{for all }f, g \in R_\circ[X].
\]
By \cite[Corollary 2.29]{Fa-Tr18}, it follows that $R_\circ[X]$ is a BF-monoid; so, in particular, $R_\circ[X]$ is acyclic (Proposition \ref{prop:BF-implies-acyclic}), although we know that this holds more generally without $R$ being a field (Example \ref{exa:polynomial-transformation-monoid}).

It is perhaps worth mentioning that, by Ritt's first decomposition theorem and its generalizations, $R_\circ[X]$ is actually an HF-monoid when $R$ is a field of characteristic zero, see \cite[Theorem 2.1]{KrTi15}. This kind of results have stimulated a great deal of research (e.g., in connection to the solution of certain types of Diophantine equations), and it is not implausible that bringing them under the umbrella of factorization theory could open the door to new and interesting developments.
\end{example}

\begin{proposition}\label{prop:new-from-old}
The following hold:
\begin{enumerate}[label={\rm (\roman{*})}]
	\item\label{it:prop:new-from-old(i)} If $\varphi: H \to K$ is a monoid homomorphism with $K$ acyclic and $\varphi^{-1}(K^\times) \subseteq H^\times$, then $H$ is acyclic.	
	\item\label{it:prop:new-from-old(ii)} If $K$ is a submonoid of an acyclic monoid $H$ and $K \cap H^\times \subseteq K^\times$, then $K$ is acyclic.
	\item\label{it:prop:new-from-old(iii)} Products and coproducts of acyclic monoids are acyclic.
\end{enumerate}
\end{proposition}

\begin{proof}
\ref{it:prop:new-from-old(i)} Let $\varphi: H \to K$ be a monoid homomorphism, and suppose that  $uxv = x$ for some $u, v, x \in H$. Then $\varphi(u) \, \varphi(x) \, \varphi(v) = \varphi(x)$, and if $K$ is acyclic, this is only possible if $\varphi(u), \varphi(v) \in K^\times$. Therefore, if $K$ is acyclic and $\varphi^{-1}(K^\times) \subseteq H^\times$, then $u, v \in H^\times$. So we are done, since $u$, $v$, and $x$ were arbitrary.
	
\vskip 0.1cm
\ref{it:prop:new-from-old(ii)} This is straightforward from part \ref{it:prop:new-from-old(i)}, when considering that, if $K$ is a submonoid of a monoid $H$, then the inclusion map $\imath: K \to H: x \mapsto x$ is a monoid homomorphism with $\imath^{-1}(H^\times) = K^\times$. 

\vskip 0.1cm
\ref{it:prop:new-from-old(iii)} Let $H$ and $K$ be, resp., the product and coproduct of an indexed family $(H_i)_{i \in I}$ of acyclic monoids.
As is usual, we will regard an element of $H$ as a function $f: I \to \bigcup_{i \in I} H_i$ such that $f(i) \in H_i$ for each $i \in I$, and an element of $K$ as a function $f \in H$ such that $f(i) \ne 1_{H_i}$ for all but finitely many $i \in I$. Accordingly, to multiply two elements $f, g \in H$ is to consider the function $fg: I \mapsto \bigcup_{i \in I} H_i: i \mapsto f(i) g(i)$. 
Then the identity of $H$ is the function $I \to \bigcup_{i \in I} H_i: i \mapsto 1_{H_i}$, and $f \in H$ is a unit if and only if $f(i) \in H_i^\times$ for every $i \in I$; also, $K$ is obviously a submonoid of $H$ with $K \cap H^\times \subseteq K^\times$. So, by part \ref{it:prop:new-from-old(ii)}, it suffices to check that $H$ is acyclic. So, assume $ufv = f$ for some $u, f, v \in H$. Then $u(i)\, f(i) \, v(i) = f(i)$ for each $i \in I$, and this yields $u(i), v(i) \in H_i^\times$ (by the hypothesis that $H_i$ is acyclic). Thus $u, v \in H^\times$, and we see that $H$ is acyclic (as wished).
\end{proof}
Incidentally, Proposition \ref{prop:BF-implies-acyclic} is a strengthening of \cite[Corollary 2.29]{Fa-Tr18}, with ``acyclic'' replacing ``unit-cancellative'' (recall from Remark \ref{rem:history} that any unit-cancellative monoid is acyclic, but not conversely).

With the above in place, we now turn to prove the main results of this section (Corollaries \ref{cor:lfgu-acyclic-mons-are-atomic} and \ref{cor:comm-unit-canc-lfgu-mons-are-FF}). 
We start with some elementary properties of acyclic monoids.

\begin{lemma}\label{lem:dedekind-finite}
	Let $H$ be an acyclic monoid. The following hold:
	\begin{enumerate}[label=\textup{(\roman{*})}]
		\item\label{lem:dedekind-finite(i)} If $x_1 \cdots x_n \in H^\times$ for some $x_1, \ldots, x_n \in H$, then $x_i \in H^\times$ for each $i \in \llb 1, n \rrb$.
		\item\label{lem:dedekind-finite(ii)} Every element of $H$ of finite order is a unit.
		\item\label{lem:dedekind-finite(iii)} $H$ has no non-trivial idempotent.
	\end{enumerate}
\end{lemma}

\begin{proof}
	\ref{lem:dedekind-finite(i)} This is a special case of \cite[Lemma 2.27(i)]{Fa-Tr18}, because an acyclic monoid is unit-cancellative (as already observed in Remark \ref{rem:history}). 
	
	\vskip 0.1cm
	\ref{lem:dedekind-finite(ii)} Let $x \in H$ be an element of finite order. Then $x^n = x^m$ for some $m, n \in \mathbf N^+$ with $m < n$. Since $H$ is acyclic, this implies that $x^{n-m}$ is a unit of $H$, which, by part \ref{lem:dedekind-finite(i)}, is only possible if $x \in H^\times$.
	
	\vskip 0.1cm
	\ref{lem:dedekind-finite(iii)} Let $x \in H$ be an idempotent (namely, $x^2 = x$). Then $\text{ord}_H(x) \le 2$, and we get from part \ref{lem:dedekind-finite(ii)} that $x$ is a unit. It follows that $1_H = x^{-1}x = x^{-1} x^2 = x$, and hence $H$ has no non-trivial idempotents.
\end{proof}

\begin{theorem}\label{lem:acyclic-fgu-is-atomic}
	Assume $H$ is an acyclic f.g.u. monoid. Then there exists a finite subset $\bar{A}$ of $\mathscr A(H)$ such that
	$
	H \setminus H^\times = \mathrm{Sgp}\langle H^\times \bar{A} H^\times \rangle_H$. In particular, $H$ is atomic.
\end{theorem}

\begin{proof}
	Since $H$ is an f.g.u. monoid, $H \setminus H^\times \subseteq \mathrm{Sgp}\langle H^\times A H^\times \rangle_H$ for some finite set $A \subseteq H$; and since $H^\times$ is a subgroup of $H$, we deduce from Lemma \ref{lem:minimal-generating-set} that there is a set $\bar{A} \subseteq A \setminus H^\times$ such that 
	\begin{equation}\label{equ:maximal-ideal-is-contained-in-special-subSgp}
	H \setminus H^\times \subseteq \mathrm{Sgp}\langle H^\times \bar{A} H^\times \rangle_H
	\end{equation}
	and
	\begin{equation}\label{equ:minimality-of-generating-set}
	a \notin \mathrm{Sgp}\langle H^\times \bar{A} H^\times \setminus H^\times aH^\times \rangle_H,
	\quad
	\text{for every }a \in \bar{A}.
	\end{equation} 
	We claim $\bar{A} \subseteq \mathscr A(H)$. So, suppose to the contrary that $\bar{A}$ has an element $\bar{a}$ that is not an atom of $H$. Then $\bar{a} = xy$ for some $x, y \in H \setminus H^\times$ (since $\bar A$ and $H^\times$ are disjoint, $\bar a$ cannot be a unit), and it follows by \eqref{equ:maximal-ideal-is-contained-in-special-subSgp} that there are $h, k \in \mathbf N^+$, $u_1, \ldots, u_{h+k}, v_1, \ldots, v_{h+k} \in H^\times$, and $a_1, \ldots, a_{h+k} \in \bar{A}$ for which
	\begin{equation}\label{equ:decomposition-of-x&y}
	x = u_1 a_1 v_1 \cdots u_h a_h v_h
	\quad\text{and}\quad
	y = u_{h+1} a_{h+1} v_{h+1} \cdots u_{h+k} a_{h+k} v_{h+k}.
	\end{equation}
	In turn, this yields that $a_{\hat\imath} = \bar{a}$ for some $\hat\imath \in \llb 1, h+k \rrb$, or else we would have that
	$$
	\bar{a} = xy = u_1 a_1 v_1 \cdots u_{h+k} a_{h+k} v_{h+k} \in \mathrm{Sgp} \langle H^\times \bar{A} H^\times \setminus H^\times \bar{a} H^\times \rangle_H, 
	$$
	in contradiction to \eqref{equ:minimality-of-generating-set}. Consequently, we find that 
	$
	\bar{a} = u u_{\hat\imath} \bar{a} v_{\hat\imath} v$, where 
	\[
	u := \prod_{1 \le i < \hat\imath} u_i a_i v_i
	\quad\text{and}\quad
	v :=  \prod_{\hat\imath < i \le h+k} u_i a_i v_i.
	\]
	Thus $u u_{\hat\imath}, v_{\hat\imath}v \in H^\times$ (because $H$ is acyclic), which implies, by Lemma \ref{lem:dedekind-finite}\ref{lem:dedekind-finite(i)}, that 
	\[
	u_i, a_i, v_i \in H^\times, \quad\text{for every }i \in \llb 1, h+k \rrb \setminus \{\hat\imath\}. 
	\]
	Together with \eqref{equ:decomposition-of-x&y}, this in turn shows that $x = \prod_{i=1}^h u_i a_i v_i \in H^\times$ if $h < \hat\imath$, and $y = \prod_{i=h+1}^{h+k} u_i a_i v_i \in H^\times$ if $\hat\imath \le h$. 
	So we got a contradiction (by hypothesis, neither $x$ nor $y$ is a unit of $H$), and we conclude that $\bar{A}$ is contained in $\mathscr A(H)$ (as wished).
	
	By \eqref{equ:maximal-ideal-is-contained-in-special-subSgp}, it only remains to check that $\mathrm{Sgp}\langle H^\times \bar{A} H^\times \rangle_H \subseteq H \setminus H^\times$, which is however straightforward, by the fact that $uav \in \mathscr A(H)$ for all $u, v \in H^\times$ and $a \in \mathscr A(H)$, see Lemma \ref{lem:2.4}\ref{lem:2.4(i)}.
\end{proof}
We note in passing that Theorem \ref{lem:acyclic-fgu-is-atomic} is a generalization (from cancellative commutative f.g.u. monoids to acyclic f.g.u. monoids) of \cite[Proposition 1.1.7.2]{GeHK06}.

\begin{corollary}
\label{cor:lfgu-acyclic-mons-are-atomic}
	Every acyclic l.f.g.u. monoid is \textup{FmF} \textup{(}and in particular atomic\textup{)}.
\end{corollary}

\begin{proof}
	This is an immediate consequence of Proposition \ref{prop:divisor-closed-sub}\ref{it:prop:divisor-closed-sub(iii)}, Theorem \ref{lem:acyclic-fgu-is-atomic}, and Corollary \ref{cor:lfgu-is-FmF}.
\end{proof}
The next result is a strengthening (with ``FF'' replacing ``BF'') of \cite[Proposition 3.4]{FGKT}.

\begin{corollary}\label{cor:comm-unit-canc-lfgu-mons-are-FF}
	Every commutative, acyclic, l.f.g.u. monoid is \textup{FF}.
\end{corollary}

\begin{proof}
	From Corollary \ref{cor:lfgu-acyclic-mons-are-atomic}, every  acyclic l.f.g.u. monoid is FmF; and from Remark \ref{rem:history}, a commutative monoid is acyclic if and only if it is unit-cancellative. So we are done, since all factorizations (into atoms) in a commutative unit-cancellative monoid are minimal, see \cite[Proposition 4.6(v)]{An-Tr18}.
\end{proof}
Now we show, as a complement to Theorem \ref{lem:acyclic-fgu-is-atomic}, that cancellative f.g. monoids, on the one hand, need not be atomic; and, on the other hand, can be atomic without being acyclic (we invite the reader to review \S\S{ }\ref{sec:generalities} and \ref{sec:presentations} before reading further).

\begin{example}
\label{exa:non-atomic-2-generator-1-relator-canc-mon}
	Fix $n \in \mathbf N^+$, and let $H$ be the monoid defined by the presentation $\mathrm{Mon}\langle X \mid R \rangle$, where $X$ is the $2$-element set $\{x, y\}$ and $R := \{(x^{\ast n}, y * x^{\ast n} * y)\} \subseteq \mathscr F(X) \times \mathscr F(X)$. 
	
	Of course, $H$ is f.g., and since the presentation $\mathrm{Mon}\langle X \mid R \rangle$ is finite and its left and right graphs are cycle-free, we get from Adian's embedding theorem (viz., Theorem \ref{th:adian-theorem}) that $H$ embeds into a group and is therefore cancellative. Also, it is straightforward that $H$ is reduced, $\{y\} \subseteq \mathscr A(H) \subseteq X$, and $x \notin \langle y \rangle_H$. Accordingly, we see that $H$ is atomic if and only if $\mathscr A(H) = X$, and the latter holds if and only if $n \ge 2$ (if $n = 1$, then $x \equiv y \ast x \ast y \bmod R^\sharp$, implying that $x$ is not an atom of $H$). 
	
Notice that $H$ is not acyclic and, hence, Corollary \ref{cor:lfgu-acyclic-mons-are-atomic} does not apply. However, it follows from the above and Corollary \ref{cor:lfgu-is-FmF} that, for $n \ge 2$, $H$ is FmF, although not BF (in fact, a routine induction shows that $x^{\ast n} \equiv y^{\ast k} \ast x^{\ast n} \ast y^{\ast k} \bmod R^\sharp$, and hence $n+2k \in \mathsf L_H(x^{\ast n})$, for every $k \in \mathbf N$).
\end{example}
Finally, we prove that acyclic f.g. monoids, although atomic (by Corollary \ref{cor:lfgu-acyclic-mons-are-atomic}), need not be BF, not even under the additional condition of being cancellative and reduced (in stark contrast to what happens in the commutative case, cf. \cite[Proposition 2.7.8.4]{GeHK06} and Corollary \ref{cor:comm-unit-canc-lfgu-mons-are-FF}): Note that one difficulty in the construction of a monoid with these characteristics lies in the fact that, by Lemma \ref{lem:dedekind-finite}\ref{lem:dedekind-finite(iii)}, acyclic monoids have no non-trivial idempotents. 
\begin{example}
\label{exa:atomic-2-gen-1-rel-canc-monoid}
Let $X$ be the $4$-element set $\{w,x,y,z\}$, and for each $k \in \mathbf N$ let $\mathfrak a_k$ denote the $X$-word $x \ast y^{\ast k} \ast z$.
Then take $H$ to be the monoid defined by the presentation $\mathrm{Mon}\langle X \mid R \rangle$, where 
$$
R := \{(\mathfrak a_k, y \ast \mathfrak a_{k+1} \ast w): k \in \mathbf N\} \subseteq \mathscr F(X) \times \mathscr F(X). 
$$
Since $\|\mathfrak a_k\|_X = k + 2 \ge 2$ for every $k \in \mathbf N$, it is routine to check that $H$ is a reduced monoid with $\mathscr A(H) = X$. Moreover, $\mathsf L_H(\mathfrak a_0)$ contains the set $\{3k: k \in \mathbf N^+\}$, because we have (by induction) that
\[
\mathfrak a_0 \equiv y^{\ast (k-1)} \ast \mathfrak a_k \ast w^{\ast (k-1)} \bmod R^\sharp,
\quad\text{for every }k \in \mathbf N^+. 
\]
Thus $H$ is atomic but not BF, and we are going to show that it is also cancellative and acyclic. 

So, denote by $\psi$ the function $\mathscr F(X) \to \mathbf N$ that maps an $X$-word $\mathfrak z$ to the supremum of the set of all $\ell \in \mathbf N^+$ for which there exist $\hat{\mathfrak z}_0, \ldots, \hat{\mathfrak z}_\ell \in \mathscr F(X)$ and $s_1, \ldots, s_\ell \in \mathbf N$ such that
\[
\mathfrak z = \hat{\mathfrak z}_0 \ast \mathfrak a_{s_1} \ast \hat{\mathfrak z}_1 \ast \cdots \ast \hat{\mathfrak z}_{\ell-1} \ast \mathfrak a_{s_\ell} \ast \hat{\mathfrak z}_\ell = \Conv_{i=1}^\ell (\hat{\mathfrak z}_{i-1} \ast \mathfrak a_{s_i} \ast \hat{\mathfrak z}_i),
\]
where we make the convention that $\sup \emptyset := 0$. Since 
\begin{equation}
\label{equ:omega-of-defining-rels}
\psi(\mathfrak a_k) = \psi(y \ast \mathfrak{a}_{k+1} \ast w) = 1,
\quad\text{for every }
k \in \mathbf N, 
\end{equation}
it is easily seen that
\begin{equation}
\label{equ:omega-of-congruent-words}
\psi(\mathfrak b) = \psi(\mathfrak c), \quad\text{for all }\mathfrak b, \mathfrak c \in \mathscr F(X) \ \text{with } \mathfrak b \equiv \mathfrak c \bmod R^\sharp.
\end{equation}
Let $\mathfrak z \in \mathscr F(X)$ and set $\ell := \psi(\mathfrak z)$. If $\ell$ is not zero, it follows from the above that there are determined $r_1, s_1, t_1, \ldots, r_\ell, s_\ell, t_\ell \in \mathbf N$ and $\mathfrak z_0, \ldots, \mathfrak z_\ell \in \mathscr F(X)$ such that $y$ is not the right-most letter of any of the words $\mathfrak z_0, \ldots, \mathfrak z_{\ell-1}$; $w$ is not the left-most letter of any of the words $\mathfrak z_1, \ldots, \mathfrak z_\ell$; and $\mathfrak z$ can be written as
\begin{equation}
\label{equ:canonical-decomposition}
\mathfrak z_0 \ast y^{\ast r_1} \ast \mathfrak a_{s_1} \ast w^{\ast t_1} \ast \mathfrak z_1 \ast \cdots \ast \mathfrak z_{\ell-1} \ast y^{\ast r_\ell} \ast \mathfrak a_{s_\ell} \ast w^{\ast t_\ell} \ast \mathfrak z_\ell.
\end{equation}
In addition, it is clear from the definition of $\psi$ that the words $\mathfrak z_0, \ldots, \mathfrak z_\ell$ in this representation satisfy
\begin{equation}\label{equ:19}
\psi(\mathfrak z_i) = 0, \quad \text{for every }i \in \llb 0, \ell \rrb.
\end{equation}
Accordingly, we say $\mathfrak z$ is \emph{normal} (with respect to the presentation $H$) if either $\ell$ is zero, or $\ell$ is non-zero and, for each $i \in \llb 1, \ell \rrb$, at least one of $r_i$, $s_i$, and $t_i$ in the decomposition \eqref{equ:canonical-decomposition} of $\mathfrak z$ is zero. 

Every $X$-word $\mathfrak z$ is $R^\sharp$-congruent to a unique normal $X$-word, which we denote by $\underline{\mathfrak z}$ and refer to as the \emph{normal form} of $\mathfrak z$ (relative to the presentation $H$). Indeed, set $\ell := \psi(\mathfrak z)$. If $\ell = 0$, then we obtain from \eqref{equ:omega-of-defining-rels} that $\underline{\mathfrak z} = \{\mathfrak z\}$, and there is nothing left to prove. Otherwise, assuming $\mathfrak z$ to be written as in \eqref{equ:canonical-decomposition}, it is  found that an $X$-word $\mathfrak y$ is congruent to $\mathfrak z$ modulo $R^\sharp$ if and only if $\mathfrak y$ can be written as
\[
\mathfrak z_0 \ast y^{\ast (r_1 + k_1)} \ast \mathfrak a_{s_1 + k_1} \ast w^{\ast (t_1 + k_1)} \ast \mathfrak z_1 \ast \cdots \ast \mathfrak z_{\ell-1} \ast y^{\ast (r_\ell+k_\ell)} \ast \mathfrak a_{s_\ell+k_\ell} \ast w^{\ast (t_\ell+k_\ell)} \ast \mathfrak z_\ell,
\]
where $k_1, \ldots, k_\ell \in \mathbf Z$ and $k_i \ge - \min(r_i, s_i, t_i)$ for each $i \in \llb 1, \ell \rrb$: The proof is by induction on the length of an $R$-chain from $\mathfrak z$ to $\mathfrak y$. The induction basis is trivial (an $R$-chain of length $0$ from $\mathfrak z$ to $\mathfrak y$ implies $\mathfrak z = \mathfrak y$). The inductive step comes down to observing that, if an $X$-word of the form 
\[
\mathfrak z_0 \ast y^{\ast \alpha_1} \ast \mathfrak a_{\beta_1} \ast w^{\ast \gamma_1} \ast \mathfrak z_1 \ast \cdots \ast \mathfrak z_{\ell-1} \ast y^{\ast \alpha_\ell} \ast \mathfrak a_{\beta_\ell} \ast w^{\ast \gamma_\ell} \ast \mathfrak z_\ell,
\]
with $\alpha_i, \beta_i, \gamma_i \in \mathbf N$ for each $i \in \llb 1, \ell \rrb$, factors as $\mathfrak p \ast \mathfrak q \ast \mathfrak r$ for some $\mathfrak p, \mathfrak q, \mathfrak r \in \mathscr F(X)$ such that $\mathfrak q$ is a defining relation of $H$, then there exist $d \in \llb 1, \ell \rrb$ and $k \in \llb 0, \min(1, \alpha_d, \gamma_d) \rrb$ for which
\[
\textstyle \mathfrak p = \mathfrak z_0 \ast \mathfrak p^\prime \ast y^{\ast (\alpha_d - k)}, 
\quad \mathfrak q = y^{\ast k} \ast \mathfrak a_{\beta_d} \ast w^{\ast k}, 
\quad\text{and}\quad
\mathfrak r = w^{\ast (\gamma_{d} - k)} \ast \mathfrak r^\prime \ast \mathfrak z_\ell,
\]
where 
\[ 
\mathfrak p^\prime := \Conv_{i=1}^{d-1}  (y^{\ast \alpha_i} \ast \mathfrak a_{\beta_i} \ast w^{\ast \gamma_i} \ast \mathfrak z_i)
\quad\text{and}\quad
\mathfrak r^\prime := \Conv_{i=d+1}^{\ell\;} (\mathfrak z_{i-1} \ast y^{\ast \alpha_{i}} \ast \mathfrak a_{\beta_{i}} \ast w^{\ast \gamma_{i}});
\]
in particular, note that $\psi(y^{\ast \alpha} \ast \mathfrak a_\beta \ast w^{\ast \gamma}) = 1$ for all $\alpha, \beta, \gamma \in \mathbf N$, while it is immediate from \eqref{equ:19} that
\[
\psi(\mathfrak z_{i-1} \ast y^{\ast \alpha_i} \ast x \ast y^{\ast \beta_i}) = \psi(y^{\ast \beta_i} \ast z \ast w^{\ast \gamma_i} \ast \mathfrak z_i) = 0,
\quad\text{for each }i \in \llb 1, \ell \rrb.
\]
Knowing now that every element of $H$ has a unique normal form, we proceed to prove a series of claims.

\begin{claim}\label{claim:A}
Let $\mathfrak z \in \mathscr F(X)$ and $a \in X$, and assume $\mathfrak z$ is normal but $a \ast \mathfrak z$ \textup{(}resp., $\mathfrak z \ast a$\textup{)} is not. Then $a = y$ \textup{(}resp., $a = w$\textup{)} and there exist $s, t \in \mathbf N^+$ and $\mathfrak z^\prime \in \mathscr F(X)$ such that 
$\mathfrak z = \mathfrak a_s \ast w^{\ast t} \ast \mathfrak z^\prime$ \textup{(}resp., $\mathfrak z = \mathfrak z^\prime \ast y^{\ast s} \ast \mathfrak a_t$\textup{)} and $\mathfrak z^\prime$ is a normal word whose left-most letter is not equal to $w$ \textup{(}resp.,  $y$\textup{)}; moreover, the normal form of $a \ast \mathfrak z$ \textup{(}resp., $\mathfrak z \ast a$\textup{)} is the $X$-word $\mathfrak a_{s-1} \ast w^{\ast (t-1)} \ast \mathfrak z^\prime$ \textup{(}resp., $\mathfrak z^\prime \ast y^{\ast (s-1)} \ast \mathfrak a_{t-1}$\textup{)}.
\end{claim}

\begin{proof}
Set $\ell := \psi(\mathfrak z)$. We will prove the statement for $a \ast \mathfrak z$; the one for $\mathfrak z \ast a$ can be proved in a similar fashion (we will omit the details).
It is obvious that $\psi(a \ast \mathfrak z) \le \ell + 1$; and on the other hand, it is clear that $\psi(a \ast \mathfrak z)$ is not zero, or else $a \ast \mathfrak z$ would be normal. We want to show that $\ell \ne 0$.

Suppose to the contrary that $\ell$ is zero. Then we see from the above that $\psi(a \ast \mathfrak z) = 1$, which in turn implies that $a \ast \mathfrak z = \mathfrak u_0 \ast y^{\ast \alpha} \ast \mathfrak a_\beta \ast w^{\ast \gamma} \ast \mathfrak u_1$ for some $\alpha, \beta, \gamma \in \mathbf N$ and $\mathfrak u_0, \mathfrak u_1 \in \mathscr F(X)$ such that $y$ is not the right-most letter of $\mathfrak u_0$ and $w$ is not the left-most letter of $\mathfrak u_1$. But this is only possible if $\alpha$, $\beta$, and $\gamma$ are \emph{positive} integers, because $a \ast \mathfrak z$ is not normal. Therefore, $\mathfrak u_0 \ast y^\alpha = a \ast \mathfrak u_0^\prime$ for some $\mathfrak u_0^\prime \in \mathscr F(X)$, which is however a contradiction, as it implies that $\psi(\mathfrak z) = \psi(\mathfrak u_0^\prime \ast \mathfrak a_\beta \ast w^{\ast \gamma} \ast \mathfrak u_1) = 1$.

Having established that $\ell \ne 0$ and taking into consideration that $\mathfrak z$ is normal, we can now write $\mathfrak z$ as in \eqref{equ:canonical-decomposition},
with the additional restriction that $r_i s_i t_i = 0$ for each $i \in \llb 1, \ell \rrb$. In particular, this means that 
\begin{equation}\label{equ:rep-of-a*z}
\mathfrak z = \mathfrak z_0 \ast y^{\ast r} \ast \mathfrak a_s \ast w^{\ast t} \ast \mathfrak z^\prime
\end{equation}
for some $r, s, t \in \mathbf N$ and $\mathfrak z_0, \mathfrak z^\prime \in \mathscr F(X)$ such that $rst = \psi(\mathfrak z_0) = 0$, the right-most letter of $\mathfrak z_0$ is not $y$, and $\mathfrak z^\prime$ is a normal $X$-word whose left-most letter is not $w$. We aim to prove that $\mathfrak z_0 = \varepsilon$.

Assume to the contrary that $\mathfrak z_0 \ne \varepsilon$. Then $\psi(a \ast \mathfrak z_0) = 1$, or else it would follow from the above that $a \ast \mathfrak z$ is normal (a contradiction). We thus have $a \ast \mathfrak z_0 = \mathfrak v_0 \ast y^{\ast h} \ast \mathfrak a_j \ast w^{\ast k} \ast \mathfrak v_1$ for some $j, h, k \in \mathbf N$ and $\mathfrak v_0, \mathfrak v_1 \in \mathscr F(X)$ such that $y$ is not the right-most letter of either $\mathfrak v_0$ or $\mathfrak v_1$; $w$ is not the left-most letter of $\mathfrak v_1$; and $\psi(\mathfrak v_0) = \psi(\mathfrak v_1) = 0$. However, this can only happen if $\mathfrak v_0 = \varepsilon$ and $h = 0$, or else $\mathfrak v_0 \ast y^{\ast h} = a \ast \mathfrak v^\prime$ for some $\mathfrak v^\prime \in \mathscr F(X)$ and, hence, $\psi(\mathfrak z_0) = \psi(\mathfrak v^\prime \ast \mathfrak a_j \ast w^{\ast k} \ast \mathfrak v_1) = 1$ (again a contradiction). So we find that
\[
a \ast \mathfrak z = \mathfrak a_j \ast w^{\ast k} \ast \mathfrak v_1 \ast y^{\ast r} \ast \mathfrak a_s \ast w^{\ast t} \ast \mathfrak z^\prime,
\]
which is still impossible, as it implies that $a \ast \mathfrak z$ is normal (on account of the properties of $\mathfrak v_1$ and $\mathfrak z^\prime$). 

Putting it all together, we can thus conclude that $\mathfrak z_0 = \varepsilon$ (as wished), and we obtain from  \eqref{equ:rep-of-a*z} that
\[
\mathfrak z = y^{\ast r} \ast \mathfrak a_s \ast w^{\ast t} \ast \mathfrak z^\prime
\quad\text{and}\quad
a \ast \mathfrak z = a \ast y^{\ast r} \ast \mathfrak a_s \ast w^{\ast t} \ast \mathfrak z^\prime.
\]
So the only way for $a \ast \mathfrak z$ not to be normal is if $a = y$ and $s, t \in \mathbf N^+$. But then $r = 0$ (recall that $rst = 0$), and we find that $\mathfrak z = \mathfrak a_s \ast w^{\ast t} \ast \mathfrak z^\prime$ and $a \ast \mathfrak z = y \ast \mathfrak a_s \ast w^{\ast t} \ast \mathfrak z^\prime$. This finishes the proof, since it is evident that the normal form of $y \ast \mathfrak a_s \ast w^{\ast t} \ast \mathfrak z^\prime$ is the $X$-word $\mathfrak a_{s-1} \ast w^{\ast (t-1)} \ast \mathfrak z^\prime$.
\end{proof}

\begin{claim}\label{claim:special-R-chain}
Suppose that $a \ast \mathfrak u \equiv a \ast \mathfrak v \bmod R^\sharp$ \textup{(}resp., $\mathfrak u \ast a \equiv \mathfrak v \ast a \bmod R^\sharp$\textup{)} for some $\mathfrak u, \mathfrak v \in \mathscr F(X)$ and $a \in X$. Then $\mathfrak u \equiv \mathfrak v \bmod R^\sharp$.
\end{claim}

\begin{proof}
	Every $X$-word is $R^\sharp$-congruent to its own normal form, and since the normal form is unique, we see that $\mathfrak z \equiv \mathfrak y \bmod R^\sharp$ if and only if $\underline{\mathfrak z} = \underline{\mathfrak y}$. Therefore, it will be enough to assume $\mathfrak u = \underline{\mathfrak u}$ and $\mathfrak v = \underline{\mathfrak v}$, and to show that $a \ast \mathfrak u \equiv a \ast \mathfrak v \bmod R^\sharp$ implies $\mathfrak u = \mathfrak v$; the ``symmetric'' statement that $\mathfrak u \ast a \equiv \mathfrak v \ast a \bmod R^\sharp$ implies $\mathfrak u = \mathfrak v$, can be proved in essentially the same way (we will omit the details).
	
	So, suppose $a \ast \mathfrak u \equiv a \ast \mathfrak v \bmod R^\sharp$, or equivalently $\underline{a \ast \mathfrak u} = \underline{a \ast \mathfrak v}$. If $a \ast \mathfrak u$ and $a \ast \mathfrak v$ are both in normal form, then $a \ast \mathfrak u = a \ast \mathfrak v$ and we are done. So, assume without loss of generality that $a \ast \mathfrak u$ is not normal. Then we have from Claim \ref{claim:A} that $a$ is equal to $y$ and there exist $s, t \in \mathbf N^+$ and $\mathfrak u^\prime \in \mathscr F(X)$ such that 
	\begin{equation}\label{equ:21}
	\mathfrak u = \mathfrak a_s \ast w^{\ast t} \ast \mathfrak u^\prime
	\quad\text{and}\quad
	\underline{a \ast \mathfrak u} = \mathfrak a_{s-1} \ast w^{\ast (t-1)} \ast \mathfrak u^\prime.
	\end{equation}
	We thus see that $a \ast \mathfrak v$, too, is not normal; otherwise, it would follow from the above that
	\[
	\mathfrak a_{s-1} \ast w^{\ast (t-1)} \ast \mathfrak u^\prime = a \ast \mathfrak u = \underline{a \ast \mathfrak v} = a \ast \mathfrak v = y \ast \mathfrak v,
	\]
	which is impossible, because the left-most letter of $\mathfrak a_s$ is $x$. Then, again by Claim \ref{claim:A}, we find that 
	\begin{equation}\label{equ:22}
	\mathfrak v = \mathfrak a_\beta \ast w^{\ast \gamma} \ast \mathfrak v^\prime
	\quad\text{and}\quad
	\underline{a \ast \mathfrak v} = \mathfrak a_{\beta-1} \ast w^{\ast (\gamma-1)} \ast \mathfrak v^\prime
	\end{equation}
	for some $\alpha, \beta \in \mathbf N^+$ and $\mathfrak v^\prime \in \mathscr F(X)$. Since $\underline{a \ast \mathfrak u} = \underline{a \ast \mathfrak v}$, we thus conclude that 
	\[
	\mathfrak a_{s-1} \ast w^{\ast (t-1)} \ast \mathfrak u^\prime = \mathfrak a_{\beta-1} \ast w^{\ast (\gamma-1)} \ast \mathfrak u^\prime,
	\]
	which is only possible if $s = \beta$, $t = \gamma$, and $\mathfrak u^\prime = \mathfrak v^\prime$ (on account of the properties of $\mathfrak u^\prime$ and $\mathfrak v^\prime$). This suffices to complete the proof, as it shows, by \eqref{equ:21} and \eqref{equ:22}, that $\mathfrak u = \mathfrak v$. 
\end{proof}
Finally, we come to the conclusions. Let $\mathfrak u$, $\mathfrak v$, and $\mathfrak z$ be arbitrary $X$-words. If $\mathfrak z \ast \mathfrak u \equiv \mathfrak z \ast \mathfrak v \bmod R^\sharp$ or $\mathfrak u \ast \mathfrak z \equiv \mathfrak v \ast \mathfrak z \bmod R^\sharp$, then Claim \ref{claim:special-R-chain} implies, by induction on $\|\mathfrak z\|_X$, that $\mathfrak u \equiv \mathfrak v \bmod R^\sharp$. This proves that $H$ is cancellative, and we are left to verify that $H$ is acyclic. So, set $\ell := \psi(\mathfrak z)$ and suppose that 
\begin{equation}\label{equ:4.23}
\mathfrak z \equiv \mathfrak u \ast \mathfrak z \ast \mathfrak v \bmod R^\sharp.
\end{equation}
We need to prove $\mathfrak u = \mathfrak v = \varepsilon$. 
To begin, it is clear that, for \eqref{equ:4.23} to hold, $\mathfrak u$ and $\mathfrak v$ must be $\{y,w\}$-words, because the $X$-words $\mathfrak a_k$ and $y \ast \mathfrak a_{k+1} \ast w$ contain, for every $k \in \mathbf N$, an equal number of $x$'s and $z$'s, and this implies that the same is true for all $X$-words $\mathfrak b$ and $\mathfrak c$ with $\mathfrak b \equiv \mathfrak c \bmod R^\sharp$. If, on the other hand, $\ell$ is zero, then we have from  \eqref{equ:omega-of-congruent-words} that $\mathfrak z$ and $\mathfrak u \ast \mathfrak z \ast \mathfrak v$ are both in normal form; whence \eqref{equ:4.23} is only possible if $\mathfrak u = \mathfrak v = \varepsilon$, and we are done. Therefore, we assume from now on that $\ell \ne 0$; and since $\mathfrak z \equiv \underline{\mathfrak z} \bmod R^\sharp$ and $R^\sharp$ is a monoid congruence, we may also assume without loss of generality that $\mathfrak z$ is normal. 
Accordingly, let $\mathfrak z$ be written as in \eqref{equ:canonical-decomposition},
with the additional restriction that at least one of the exponents $r_i$, $s_i$, and $t_i$ is zero for each $i \in \llb 1, \ell \rrb$. In consequence, it is not difficult to see that \eqref{equ:4.23} can only be true if $\mathfrak z_0 = \mathfrak z_\ell = \varepsilon$ and there exist $\alpha, \beta \in \mathbf N$ such that $\mathfrak u = y^{\ast \alpha}$ and $\mathfrak v = w^{\ast \beta}$. Then it is straightforward to determine the normal form of $\mathfrak u \ast \mathfrak z \ast \mathfrak v$ and to see that it coincides with the one of $\mathfrak z$ only if $\alpha = \beta = 0$ (as wished).
\end{example}

\section{Primes and unique factorization}

In this final section, we prove a characterization of factorial monoids based on a non-standard yet natural generalization of the notion of prime element to the non-commutative setting (Definition \ref{def:powerful-atoms}). We start by recalling the following, cf. \cite[Definition 1.1.1.3]{GeHK06}:

\begin{definition}\label{def:prime}
	Let $H$ be a monoid. An element $p \in H$ is \emph{prime} provided that $p \notin H^\times$ and if $p \mid_H xy$, for some $x, y \in H$, then $p \mid_H x$ or $p \mid_H y$. We also refer to the prime elements of $H$ as the primes of $H$.
\end{definition}

Prime elements are of great importance in many aspects of factorization theory. Among other things, it is a simple exercise to show that, in the classical setting of cancellative commutative monoids, every prime is an atom, see, e.g., \cite[Proposition 1.1.2.3]{GeHK06}; but this need no longer be true even in the ``slightly less restrictive'' case of the multiplicative monoid of an integral domain:

\begin{example}\label{exa:5.2}
	Assume $H$ is a non-trivial monoid with an absorbing element $0_H$ (meaning that $0_H x = x\, 0_H = 0_H$ for every $x \in H$) and no zero divisors (i.e., $xy \ne 0_H$ for all $x, y \in H$ with $x \ne 0_H \ne y$). Then $0_H = 0_H^2$ and $0_H \notin H^\times$, so $0_H$ is neither a unit nor an atom. Yet, $0_H$ is a prime, because $0_H \mid_H xy$ for some $x,y \in H$ only if $x$ or $y$ equals $0_H$.
\end{example}

However, it turns out that \cite[Proposition 1.1.2.3]{GeHK06} carries over to acyclic or atomic monoids.

\begin{proposition}\label{prop:primes}
Let $H$ be a monoid, and let $p$ be a prime element of $H$. The following hold:
\begin{enumerate}[label={\rm (\roman{*})}]
\item\label{prop:primes(i)} If $p \mid_H x_1 \cdots x_n$ for some $x_1, \ldots, x_n \in H$, then $p \mid_H x_i$ for some $i \in \llb 1, n \rrb$. 
\item\label{prop:primes(ii)} If $H$ is acyclic or atomic, then $p$ is an atom of $H$.
\end{enumerate}
\end{proposition}

\begin{proof} Part \ref{prop:primes(i)} is a routine induction (if $n \ge 2$ and $p \mid_H x_1 \cdots x_n$ for some $x_1, \ldots, x_n \in H$, then $p \mid_H x_1$ or $p \mid_H x_2 \cdots x_{n-1}$). Therefore, we can restrict attention to proving part \ref{prop:primes(ii)}. 

\vskip 0.1cm
\textsc{Case 1:} $H$ is acyclic. Suppose by way of contradiction that $p$ is not an atom. Then $p=xy$ for some $x,y \in H \setminus H^\times$ (recall that, by definition, a prime is not a unit); and since $p \mid_H xy$, we have that $p \mid_H x$ or $p \mid_H y$. Assume, e.g., that $x = upv$ for some $u, v \in H$ (the other case is similar). Then $p = xy = upvy$, and because $H$ is acyclic, this implies that $vy \in H^\times$. Therefore, we see from Proposition \ref{lem:dedekind-finite}\ref{lem:dedekind-finite(i)} that $y$ is a unit. So we have a contradiction, implying that $p$ is an atom.
	
\vskip 0.1cm
\textsc{Case 2:} $H$ is atomic. By definition, $p$ is not a unit of $H$, and hence there are $a_1, \ldots, a_n \in \mathscr A(H)$ such that $p = a_1 \cdots a_n$ (every non-unit element in an atomic monoid is a non-empty product of atoms). By part \ref{prop:primes(i)}, this yields that $p \mid_H a_i$ for some $i \in \llb 1, n \rrb$, and hence $a_i = u p v$ for some $u, v \in H$. On the other hand, we have by Lemma \ref{lem:2.4}\ref{lem:2.4(ii)} that neither $up$ nor $pv$ is a unit of $H$. Since an atom cannot be written as a product of two non-units, we thus find that $u, v \in H^\times$, and hence $p = u^{-1} a_i v^{-1}$. By Lemma \ref{lem:2.4}\ref{lem:2.4(i)}, this is enough to conclude that $p$ is itself an atom.
\end{proof}

Another basic result in factorization theory is that a cancellative commutative monoid $H$ is factorial if and only if $H$ is atomic and each of its atoms is a prime, see \cite[Theorem 1.1.10.2, parts (a) and (b)]{GeHK06}. The next example shows that the same need no longer be true for cancellative monoids (the reader may want to review \S { }\ref{sec:presentations} before proceeding).

\begin{example}\label{exa:5.4}
Let $H$ be the monoid defined by the presentation $\mathrm{Mon}\langle X \mid R \rangle$, where $X$ is the $2$-element set $\{x, y\}$ and $R := \{(x \ast y \ast x, y \ast x \ast y \ast x \ast y)\} \subseteq \mathscr F(X) \times \mathscr F(X)$. 

It is routine to check that $H$ is an f.g., reduced, atomic monoid with $\mathscr A(H) = X$; and similarly as in Example \ref{exa:non-atomic-2-generator-1-relator-canc-mon}, it follows by Adian's embedding theorem that $H$ is also cancellative. In addition, it takes a moment to show that $x$ and $y$ are both primes in $H$ (i.e., every atom of $H$ is a prime). However, $H$ is not BF (let alone factorial), because $x \ast y \ast x \equiv y^{\ast k} \ast x \ast y \ast x \ast y^{\ast k} \bmod R^\sharp$ for every $k \in \mathbf N$ (by induction), and hence $\{3 + 2k: k \in \mathbf N\} \subseteq \mathsf L_H(x \ast y \ast x)$.
\end{example}

On the positive side, we will see that \cite[Theorem 1.1.10.2, parts (a) and (b)]{GeHK06} carries over to arbitrary monoids on condition of replacing primes with a special type of atoms we call ``powerful'', since they occur with the same ``exponent'' in any two factorizations of a given element. 

\begin{definition}
\label{def:powerful-atoms}
	Let $H$ be a monoid. Given $a \in \mathscr A(H)$, we denote by $\mathsf{v}_a^H$ the function $\mathscr F(\mathscr A(H)) \to \mathbf N$ that maps the empty word to $0$ and a non-empty $\mathscr A(H)$-word $a_1 \ast \cdots \ast a_n$ of length $n$ to the number of indices $i \in \llb 1, n \rrb$ such that $a_i \simeq_H a$, cf. \cite[Definition 2.10]{Fa-Tr18}; and we say $a$ is \emph{powerful} if $\mathsf{v}_a^H(\mathfrak b) = \mathsf{v}_a^H(\mathfrak c)$ for all $\mathfrak b, \mathfrak c \in \mathscr F(\mathscr A(H))$ such that $\pi_H(\mathfrak b) = \pi_H(\mathfrak c)$.
\end{definition}
The function $\mathsf{v}_a^H$ in Definition \ref{def:powerful-atoms} is, in essence, a non-commutative generalization of the $p$-adic valuation commonly associated with a prime number $p$ in the positive integers. This is transparent from some of the properties such a function satisfies, which we record in a lemma for future reference.

\begin{lemma}
\label{lem:valuation}
Let $H$ be a monoid, and let $\mathfrak b, \mathfrak c \in \mathscr F(\mathscr A(H))$. The following hold:
\begin{enumerate}[label={\rm (\roman{*})}]
	\item\label{lem:valuation(i)} $\mathsf{v}_a^H(\mathfrak b \ast \mathfrak c) = \mathsf{v}_a^H(\mathfrak b) + \mathsf{v}_a^H(\mathfrak c)$ for every $a \in \mathscr A(H)$.
	\item\label{lem:valuation(ii)} If $\mathsf{v}_a^H(\mathfrak b) \ne 0$ for some $a \in \mathscr A(H)$, then $a \mid_H \pi_H(\mathfrak b)$.
	\item\label{lem:valuation(iii)} $\mathfrak b \equiv \mathfrak c \bmod \mathscr C_H$ if and only if $\pi_H(\mathfrak b) = \pi_H(\mathfrak c)$ and $\mathsf{v}_a^H(\mathfrak b) = \mathsf{v}_a^H(\mathfrak c)$ for every $a \in \mathscr A(H)$.
\end{enumerate}
\end{lemma}

\begin{proof}
The proof is straightforward from the definitions. We leave the details to the reader.
\end{proof}

By Example \ref{exa:5.4}, an atom can be prime without being powerful. Complementarily, we will show that a powerful atom need not be prime (even in a cancellative commutative monoid).

\begin{example}\label{exa:5.7}
Fix a prime number $p \in \mathbf N^+$, and let $H$ be the submonoid of the additive group of the rational field generated by $1+1/p$ and all positive rational numbers whose denominator is coprime to $p$ (here, the denominator of a rational number $x$ is the smallest integer $n \ge 1$ such that $nx \in \mathbf Z$). In fact, $H$ is one of a class of cancellative commutative monoids (named \emph{Puiseux monoids}) that have received great attention in recent years, see \cite{Ch-Go-Go19} and references therein. 

We will write $H$ additively (as is natural to do).
It is easily seen that $H$ is a reduced monoid with identity $0 \in \mathbf Q$, and $1 + 1/p$ is the only atom of $H$. In particular, note that $1+1/p$ is the smallest element in $H$ whose denominator is not coprime to $p$; and if, on the other hand, $x$ is a non-zero element of $H$ and the denominator of $x$ is coprime to $p$, then $x/(p+1)$ is also in $H$ (with the result that $x$ is not an atom, as it ``factors'' into a sum of $p+1$ non-zero elements of $H$). 

It is thus obvious that $1 + 1/p$ is a powerful atom of $H$, because $H$ is cancellative and $\mathscr F(\mathscr A(H))$ is the free monoid over the one-element alphabet $\{1 + 1/p\}$. Yet, $1+1/p$ is not a prime of $H$. In fact, it is clear that $1 \in H$ and 
$1+1/p \mid_H p+1$. So, if $1+1/p$ were a prime element of $H$, then we would obtain from Proposition \ref{prop:primes}\ref{prop:primes(i)} that $(1 + 1/p) + x = 1$ for some $x \in H$, a contradiction.
\end{example}

The relation between prime elements and powerful atoms is further clarified by the next result.

\begin{proposition}
\label{prop:5.8}
Let $H$ be a monoid. The following hold:
\begin{enumerate}[label={\rm (\roman{*})}]
\item\label{prop:5.8(i)} If $H$ is atomic, then every powerful atom is a prime.
\item\label{prop:5.8(ii)} If $H$ is cancellative and commutative, then every prime is a powerful atom.
\end{enumerate}
\end{proposition}

\begin{proof}
\ref{prop:5.8(i)} Let $a$ be a powerful atom of $H$, and assume that $a \mid_H xy$ for some $x, y \in H$, i.e., $xy = uav$ for some $u, v \in H$. We need to prove that $a \mid_H x$ or $a \mid_H y$.
To this end, we may suppose that neither $x$ nor $y$ is a unit, or else the conclusion is trivial (if, for instance, $x \in H^\times$, then $y = x^{-1} uav$ and, hence, $a \mid_H y$). Similarly, we can admit that either $u = 1_H$ or $u \notin H^\times$; otherwise, we have that $av = u^{-1} xy$, and so it suffices to show that $a \mid_H u^{-1} x$ or $a \mid_H y$ (note that $a \mid_H u^{-1} x$ if and only if $a \mid_H x$). And likewise, there is no loss of generality in assuming that $v = 1_H$ or $v \notin H^\times$. 

Since $H$ is atomic (by hypothesis), there thus exist $\mathscr A(H)$-words $\mathfrak b$, $\mathfrak c$, $\mathfrak u$, and $\mathfrak v$ such that $
x = \pi_H(\mathfrak b)$, $y = \pi_H(\mathfrak c)$, $u = \pi_H(\mathfrak u)$, and $
v = \pi_H(\mathfrak v)$; 
in particular, notice that $\mathfrak u = \varepsilon$ when $u = 1_H$, and $\mathfrak v = \varepsilon$ when $v = 1_H$. It follows that $\mathfrak b \ast \mathfrak c$ and $\mathfrak u \ast a \ast \mathfrak v$ are both factorizations (into atoms) of $xy$; and since $a$ is a powerful atom of $H$, we can conclude by Lemma \ref{lem:valuation}\ref{lem:valuation(i)} that 
\[
\mathsf{v}_a^H(\mathfrak b) + \mathsf{v}_a^H(\mathfrak c) = \mathsf{v}_a^H(\mathfrak b \ast \mathfrak c) = \mathsf{v}_a^H(\mathfrak u \ast a \ast \mathfrak v) = \mathsf{v}_a^H(\mathfrak u) + \mathsf{v}_a^H(a) + \mathsf{v}_a^H(\mathfrak v) \ge 1.
\]
Then $\mathsf{v}_a^H(\mathfrak b) \ge 1$ or $\mathsf{v}_a^H(\mathfrak c) \ge 1$, which implies by Lemma \ref{lem:valuation}\ref{lem:valuation(ii)} that $a \mid_H x$ or $a \mid_H y$ (as wished).

\ref{prop:5.8(ii)} Assume that $H$ has a prime element $p$. By Proposition \ref{prop:primes}\ref{prop:primes(ii)}, $p$ is an atom (recall from Remark \ref{rem:history}  that every cancellative commutative monoid is acyclic); it remains to prove that $p$ is also powerful. 

Let $\mathfrak a$ and $\mathfrak b$ be $\mathscr A(H)$-words with $\pi_H(\mathfrak a) = \pi_H(\mathfrak b)$, and define $h := \mathsf{v}_p^H(\mathfrak a)$ and $k := \mathsf{v}_p^H(\mathfrak b)$. 
Since $p$ is an atom and $H$ is commutative, it is thus clear from Lemma \ref{lem:2.4} and the definition of $\mathsf{v}_p^H$ that there exist $u, v \in H^\times$ and finite subsets $\mathcal A$ and $\mathcal B$ of $\mathscr A(H)$ such that 
$
\pi_H(\mathfrak a) = p^h u \, \prod_{a \in \mathcal A} a$,
$\pi_H(\mathfrak b) = p^k v \, \prod_{b \in \mathcal B} b$, and 
$p \nmid_H c$ for every $c \in \mathcal A \cup \mathcal B$ (note that $\alpha \mid_H \beta$, for some atoms $\alpha, \beta \in H$, if and only if $\alpha \simeq_H \beta$).

By symmetry, we can now assume without loss of generality that $h \le k$. Since $\pi_H(\mathfrak a) = \pi_H(\mathfrak b)$ and $H$ is cancellative, it then follows from the above that $
\prod_{a \in \mathcal A} a = p^{k-h} u^{-1}v \,\prod_{b \in \mathcal B} b$. But this is only possible if $h = k$, or else $\mathcal A$ is non-empty and $p \mid_H a$ for some $a \in \mathcal A$ (by the fact that $p$ is prime).
\end{proof}

With all this said, we are ready for the characterization promised at the beginning of the section, which we use in Example \ref{exa:5.10} to give an ``alternative proof'' of the fundamental theorem of arithmetic.

\begin{theorem}\label{th:factoriality}
	Let $H$ be a monoid. The following are equivalent:
	\begin{enumerate}[label=\textup{(\alph{*})}]
		\item\label{th:factoriality(a)} $H$ is factorial.
		\item\label{th:factoriality(b)} $H$ is atomic and every atom of $H$ is powerful.
	\end{enumerate}
\end{theorem}

\begin{proof}
	\ref{th:factoriality(a)} $\Rightarrow$ \ref{th:factoriality(b)}: By definition, a factorial monoid is, in particular, an atomic monoid; therefore, we only have to check that every atom of $H$ is powerful. 
	To this end, suppose that $\mathscr A(H)$ is non-empty (or else the conclusion is trivial), pick $a \in \mathscr A(H)$, and let $\mathfrak b$ and $\mathfrak c$ be $\mathscr A(H)$-words such that $\pi_H(\mathfrak b) = \pi_H(\mathfrak c)$. 
	
	We need to prove $\mathsf{v}_a^H(\mathfrak b) = \mathsf{v}_a^H(\mathfrak c)$.
	If either of $\mathfrak b$ or $\mathfrak c$ is empty, then $\mathfrak b = \mathfrak c = \varepsilon$ and we are done (recall from Lemma \ref{lem:2.4}\ref{lem:2.4(i)} that a non-empty product of atoms cannot be equal to a unit). Otherwise, $H$ being factorial implies that $n := \|\mathfrak b\|_{\mathscr A(H)} = \|\mathfrak c\|_{\mathscr A(H)} \in \mathbf N^+$; whence $\mathfrak b = b_1 \ast \cdots \ast b_n$ and $\mathfrak c = c_1 \ast \cdots \ast c_n$ for some unique $b_1, c_1, \ldots, b_n, c_n \in \mathscr A(H)$. Moreover, there exists a permutation $\sigma$ of $\llb 1, n \rrb$ such that $b_1 \simeq_H c_{\sigma(1)}$, \ldots, $b_n \simeq_H c_{\sigma(n)}$. Since $\simeq_H$ is an equivalence relation on $H$, it is thus clear that
	\[
	\mathsf{v}_a^H(\mathfrak b) := \bigl|\{i \in \llb 1, n \rrb: a \simeq_H b_i\}\bigr| = \bigl|\{i \in \llb 1, n \rrb: a \simeq_H c_{\sigma(i)}\}\bigr| = \bigl|\{i \in \llb 1, n \rrb: a \simeq_H c_i \}\bigr| =: \mathsf{v}_a^H(\mathfrak c).
	\]
	\ref{th:factoriality(b)} $\Rightarrow$ \ref{th:factoriality(a)}: This is an immediate consequence of Lemma \ref{lem:valuation}\ref{lem:valuation(iii)} (recall that $H$ is factorial if, by definition, any two factorizations of an element of $H$ are $\mathscr C_H$-congruent).
\end{proof}
\begin{example}\label{exa:5.10}
Let $H$ be a cancellative, commutative monoid. A \emph{common divisor} (in $H$) of a non-empty set $X \subseteq H$ is an element $\alpha \in X$ such that $\alpha \mid_H x$ for every $x \in H$. Accordingly, $H$ is a \emph{\textup{GCD}-monoid} if every non-empty finite subset $X$ has a \emph{greatest common divisor}, namely, there exists a common divisor $\alpha$ of $X$ with the additional property that, if $\beta$ is any other common divisor of $X$, then $\beta \mid_H \alpha$. 

Now, assume $H$ is a GCD-monoid. By part ii) of the unnumbered proposition on p. 114 in \cite{HK98}, every atom of $H$ is prime; and by Proposition \ref{prop:5.8}\ref{prop:5.8(ii)}, it follows that every atom of $H$ is, in fact, powerful. All in all, we thus see from Theorem \ref{th:factoriality} and the above that, if $H$ is an atomic GCD-monoid, then $H$ is factorial and every non-unit element of $H$ is a product of primes, cf. \cite[\S{ }10.7]{HK98}. 

This implies the fundamental theorem of arithmetic, because it is almost trivial to prove (without ever mentioning primes!) that the positive integers with the usual multiplication form an atomic GCD-monoid.
\end{example}
\section{Prospects for future research}
\label{sec:6}
Unfortunately, we do not know of anything analogous to Theorem \ref{th:factoriality} for minimally factorial monoids. In this regard, it would be interesting to understand, among other things, whether the ``minimal counterparts'' of some of the arithmetic invariants used in the classical theory to measure the ``distance from factoriality'' (e.g., unions of sets of lengths) are all finite for f.g.u. monoids, as is known to happen in the commutative and unit-cancellative case, see \cite[Theorem 3.5]{Ga-Ge09} and \cite[Proposition 3.4]{FGKT}. The dichotomy implied by the following proposition counts as a first little step in this direction.

\begin{proposition}
\label{prop:6.1}
	Let $H$ be a monoid. Then either the $\preceq_H$-minimal factorizations of $H$ are bounded in length \textup{(}i.e., there exists $N \in \mathbf N$ such that $\|\mathfrak a\|_{\mathscr A(H)} \le N$ for every $\preceq_H$-minimal $\mathscr A(H)$-word $\mathfrak a$\textup{)}, or for each $k \in \mathbf N$ there exists a $\preceq_H$-minimal $\mathscr A(H)$-word whose length is equal to $k$.
\end{proposition}

\begin{proof}
Assume $\mathscr A(H)$ is non-empty (otherwise the conclusion is trivial) and there exists a positive integer $k \ge 3$ that is not the length of any $\preceq_H$-minimal factorization. Then let $\mathfrak a = a_1 \ast \cdots \ast a_n$ be an $\mathscr A(H)$-word of length $n \ge k$. We will show that $\mathfrak a$ is not $\preceq_H$-minimal; this will finish the proof, since it is easy to check that every $\mathscr A(H)$-word of length $0$, $1$, or $2$ is $\preceq_H$-minimal, see \cite[Proposition 4.6(i)]{An-Tr18}.

We proceed by induction. If $n = k$, there is nothing to do. So, let $n \ge k+1$ and suppose that $n-1$ is not the length of any $\preceq_H$-minimal $\mathscr A(H)$-word. Since $\bar{\mathfrak a} := a_1 \ast \cdots \ast a_{n-1}$ is an $\mathscr A(H)$-word of length $n-1$, there then exists a non-empty $\mathscr A(H)$-word $\mathfrak b = b_1 \ast \cdots \ast b_m$ of length $m \le n-2$ such that $\mathfrak b \prec_H \bar{\mathfrak a}$; namely, $\pi_H(\bar{\mathfrak a}) = \pi_H(\mathfrak b)$ and $b_1 \simeq_H a_{\sigma(1)}$, \ldots, $b_m \simeq_H a_{\sigma(m)}$ for some injection $\sigma: \llb 1, m \rrb \to \llb 1, n \rrb$. It is thus clear that $\mathfrak b \ast a_n \prec_H \bar{\mathfrak a} \ast a_n = \mathfrak a$, implying that $\mathfrak a$ is not $\preceq_H$-minimal (as wished); note, in particular, that $\pi_H(\mathfrak b \ast a_n) = \pi_H(\mathfrak b) \ast \pi_H(a_n) = \pi_H(\bar{\mathfrak a}) \ast \pi_H(a_n) = \pi_H(\mathfrak a)$.  
\end{proof}
Let $H$ be a monoid, and define $\kappa(H)$ be the supremum of all integers $k \ge 0$ for which there exists a $\preceq_H$-minimal $\mathscr A(H)$-word of length $k$. We have by Proposition \ref{prop:6.1} that $
\llb 0, \kappa(H) \rrb = \bigcup_{x \in H} \mathsf L_H^{\sf m}(x)$,
and it seems interesting to determine $\kappa(H)$ as $H$ ranges over some specified family of monoids. 

For instance, we know from \cite[Proposition 4.11(i)]{An-Tr18} that, if $K$ is a finite monoid and $\mathcal P_{\text{fin},1}(K)$ is the \emph{reduced power monoid} of $K$ (that is, the collection of all subsets of $K$ containing $1_K$ endowed with the operation of setwise multiplication induced by $K$), then 
\begin{equation}\label{equ:(23)}
\kappa(\mathcal P_{\text{fin},1}(K)) \le |K| - 1.
\end{equation}
Is it possible to provide a ``simple meaningful characterization'' of all finite monoids $K$ for which $\mathcal P_{\text{fin},1}(K)$ is atomic and \eqref{equ:(23)} holds as an equality? Note that, by \cite[Theorem 3.9]{An-Tr18}, $\mathcal P_{\text{fin},1}(K)$ is atomic if and only if $1_K \ne x^2 \ne x$ for every $x \in K \setminus \{1_H\}$; and by \cite[Lemma 5.5]{An-Tr18}, $\kappa(\mathcal P_{\text{fin},1}(K)) = |K| - 1$ when $K$ is a cyclic group of odd order.

\section*{Acknowledgments}
\label{subsec:acks}
Part of this paper was completed while the author was visiting Nankai University from Sep to Dec 2018 and the Institute for Mathematics and University of Graz from Jan to Feb 2019. The author is particularly grateful to Weidong Gao (Nankai University), Alfred Geroldinger (University of Graz), Siao Hong (Nankai University), Guoqing Wang (Tianjin Polytechnic University), and Hanbin Zhang (Chinese Academy of Sciences, Beijing) for their hospitality. 

\end{document}